\documentclass[12pt]{amsart}
\usepackage{amsthm,amssymb,verbatim,color}
\usepackage{graphicx}

\numberwithin{equation}{section}
\usepackage{pgf,tikz}
\usepackage{mathrsfs}
\usetikzlibrary{arrows}

\usepackage{tikz}
\usetikzlibrary{arrows,shapes,positioning}
\tikzstyle{every loop}= []

\colorlet{myGray}{gray!25}

\tikzset{my circle/.style={circle,draw=black,fill=myGray,inner
    sep=0pt,minimum size=6pt}} \tikzset{my square/.style={regular
    polygon,regular polygon sides=4,draw=black,fill=myGray,inner
    sep=0pt,minimum size=9pt}} \tikzset{my star/.style={star,star
    point ratio=2.5,draw=black,fill=myGray,inner sep=0pt,minimum
    size=9pt}} \tikzset{my triangle/.style={regular polygon,regular
    polygon sides=3,draw=black,fill=myGray,inner sep=0pt,minimum
    size=9pt}} \tikzset{my
  kite/.style={diamond,aspect=0.25,draw=black,fill=myGray,inner
    sep=1pt,minimum size=6pt}}
\tikzset{every pin/.style={pin distance=3pt,inner sep=1pt,font=\tiny}}
\tikzset{every pin edge/.style={semithick}}

\newtheorem{theorem}{Theorem}[section]
\newtheorem{corollary}[theorem]{Corollary}

\newtheorem{proposition}[theorem]{Proposition}
\newtheorem{problem}[theorem]{Question}
\newtheorem{construction}[theorem]{Construction}

\theoremstyle{definition}
\newtheorem{definition}[theorem]{Definition}
\theoremstyle{remark}
\newtheorem{remark}[theorem]{Remark}
\theoremstyle{definition}
\newtheorem{example}[theorem]{Example}
\theoremstyle{definition}

\newcommand{\C}{\mathcal{C}}
\newcommand{\G}{\mathcal{G}}
\newcommand{\LL}{\mathcal{L}}

\newcommand{\pr}{\mathbb{P}}

\begin{document}


\title{Hilbert functions of schemes of double and reduced points}

\author[E. Carlini]{Enrico Carlini}
\address[E. Carlini]{DISMA-Department of Mathematical Sciences \\
Politecnico di Torino, Turin, Italy}
\email{enrico.carlini@polito.it}

\author[M. V. Catalisano]{Maria Virginia Catalisano}
\address[M. V. Catalisano]{Dipartimento di Ingegneria Meccanica,
Energetica, Gestionale e dei
Trasporti, Universit\`a degli studi di Genova, Genoa, Italy}
\email{catalisano@dime.unige.it}

\author[E. Guardo]{Elena Guardo}
\address[E. Guardo]{Dipartimento di Matematica e Informatica\\
Universit\`a degli studi di Catania\\
Viale A. Doria, 6 \\
95100 - Catania, Italy}
\email{guardo@dmi.unict.it}

\author[A. Van Tuyl]{Adam Van Tuyl}
\address[A. Van Tuyl]{Department of Mathematics and Statistics\\
McMaster University, Hamilton, ON, Canada L8S 4L8}
\email{vantuyl@math.mcmaster.ca}


\thanks{Version: \today}


\begin{abstract}
It remains an open problem to classify the Hilbert functions of double
points in $\mathbb{P}^2$.
Given a valid Hilbert function $H$ of a zero-dimensional
scheme in $\mathbb{P}^2$, we show how to construct a set of fat points
 $Z \subseteq
\mathbb{P}^2$ of double and reduced points such that $H_Z$, the
Hilbert function of $Z$, is the same as $H$.   In other words, we
show that any valid Hilbert function $H$ of a zero-dimensional
scheme is the Hilbert function of a set a positive
number of double points and some reduced
points. For some families of valid Hilbert functions, we are also
able to show that $H$ is the Hilbert function of only double points.
In addition, we give necessary and sufficient conditions for the Hilbert
function of a scheme of a double points, or double points plus one
additional reduced point, to be the Hilbert function
of points with support on a star
configuration of lines.
\end{abstract}


\maketitle



\section{Introduction}

Throughout this paper, $k$ will denote an algebraically closed field
of characteristic zero. Let $X = \{P_1,\ldots,P_s\} \subseteq \pr^2$
be a finite set of reduced points with associated homogeneous ideal
$I_X = I_{P_1}\cap \cdots \cap I_{P_s} \subseteq R = k[x_0,x_1,x_2]$.
Given positive integers $m_1,\ldots,m_s$, we let $Z = m_1P_1 + \cdots + m_sP_s$
denote the scheme defined by
the homogeneous ideal $I_Z = I_{P_1}^{m_1} \cap \cdots \cap I_{P_s}^{m_s}$.
We refer to $Z$ as a {\it set of fat points}.    We call $m_i$ the
{\it multiplicity} of the point $P_i$;  when $m_i = 2$, we sometimes
call $P_i$ a {\it double point}.   Given a set of fat points
$Z$, the {\it support} of $Z$ is the set
${\rm Supp}(Z) = \{P_1,\ldots,P_s\}$.

Information about the  set of fat points $Z$ is encoded into
its Hilbert function.   Recall that the {\it Hilbert function} of $Z$
is the function $H_Z: \mathbb{N} \rightarrow \mathbb{N}$ defined by
\[i \mapsto \dim_k (R/I_Z)_i = \dim_k R_i - \dim_k (I_Z)_i\]
where $R_i$, respectively $(I_Z)_i$, denotes the $i$-th graded
piece of $R$, respectively $(I_Z)_i$
(see
Chapter 5 of \cite{KR} for a comprehensive introduction to Hilbert functions).
It is then natural to ask
if one can characterize what functions are the Hilbert function of
a set of fat points.  A complete characterization of the Hilbert
functions of reduced points (i.e., all the $m_i = 1$) was first
described by Geramita, Maroscia, and Roberts \cite{GMR}. However,
even in the case that all the fat points are double points, a
characterization of the Hilbert functions remains elusive (see,
for example, the surveys of Gimigliano  \cite{Gi} and
Harbourne \cite{H}).  In this paper, we contribute to
this open problem by showing that every Hilbert function of a
collection of reduced points in $\mathbb{P}^2$
is also the Hilbert function of a collection of
double points and reduced points in $\mathbb{P}^2$.   In specific
cases, we can give a sufficient condition for a numerical function
to be the Hilbert function of a scheme consisting only of double
points.

To further describe our results, we introduce some additional notation.
One way to study the
Hilbert function of $H_Z$ is to study its {\it first difference function}
(sometimes called the Castelnuovo function) which is given
by
\[\Delta H_Z(i) = H_Z(i) - H_Z(i-1) ~~\mbox{for all $i \geq 0$,
where $H_Z(-1) = 0$.}\]

When $Z$ is a zero-dimensional scheme in $\pr^2$, it can be shown
(see Remark \ref{validhf}) that all but a finite number of values
of $\Delta H_Z(i)$ are zero. Furthermore, if $\Delta H_Z(i) = 0$
for all $i \geq \sigma+1$, and if we write $\Delta H_Z =
(h_0,\ldots,h_\sigma)$ to encode all the non-zero values of
$\Delta H_Z$,  then there is an $0 < \alpha \leq \sigma$ such that
  \begin{enumerate}
  \item[$(a)$] $h_i = i+1$ if $0 \leq i < \alpha$, and
  \item[$(b)$] $h_i \geq h_{i+1}$ if $\alpha \leq i \leq \sigma$.
  \end{enumerate}
We call $\Delta H = (h_0,\ldots,h_\sigma)$ a
{\it valid Hilbert function}
of a zero-dimensional scheme
in $\mathbb{P}^2$ if $\Delta H$ satisfies conditions
$(a)$ and $(b)$.  Ideally, we want to answer the
following question:

\begin{problem}\label{mainq}
Let $\Delta H = (h_0,\ldots,h_\sigma)$ be a valid Hilbert
function. Write $\sum \Delta H = \sum_{i=0}^\sigma h_i$ as $\sum
\Delta H = 3d + r$ with $r\in\{0,1,2\}$.  Does there exist a set
$Z$ of $d$ double points and $r$ reduced points in $\mathbb{P}^2$
such that $\Delta H_Z = \Delta H$?
\end{problem}

Note that a scheme $Z$ with $d$ double points and $r$ reduced points
in $\mathbb{P}^2$
will have $\deg(Z) = 3d+r$.  Furthermore, it is known that
$H_Z(i) = \deg(Z)$ for $i \gg 0$.   This explains why we require
 $\sum \Delta H = 3d + r$.  If we could answer this question,
we could determine if a valid Hilbert
function is the Hilbert function of a set of double points.  Thus, the above
question is quite difficult.

We can ask a weaker question by simply asking if any set of double
points and reduced points can be constructed:

\begin{problem}\label{weakerq}
Let $\Delta H = (h_0,\ldots,h_\sigma)$ be a valid Hilbert function.
Can one always find
integers $d$ and $r$ where $d$ is positive and $r\geq 0$ with
$\sum \Delta H = 3d+r$
such that $H$ is the Hilbert function of a set $Z$ of $d$ double points and $r$ simple
points in $\mathbb{P}^2$?
\end{problem}

Note that if we allow $d=0$ and $r = \sum \Delta H$, then the above
question is simply asking if $\Delta H$ is
the Hilbert function of $r$ reduced points, which follows from
Geramita, Maroscia, and Roberts \cite{GMR}.  We can now view
Question \ref{mainq} as asking if the $d$ in Question \ref{weakerq} can
be taken to be the maximum allowed value.  Ideally, when trying
to answer Question \ref{weakerq}, we want to
make $d$ as large as possible.

One of the main results of this paper (Theorem \ref{procedure})
will give us a tool to answer to Question \ref{weakerq}.
Specifically,  starting with a set of double and reduced
points on a collection of
general lines in $\mathbb{P}^2$,
we describe how to ``merge'' three reduced points
to make a new scheme with one new double point and three fewer
reduced points.   Moreover, this procedure does not change
the Hilbert function. The results
of Cooper, Harbourne, and Teitler \cite{CHT} are the crucial
ingredient to prove that our new configuration has the correct
Hilbert function.    By reiterating this process, in a controlled
fashion, Construction \ref{algorithm} shows how to
start from a valid Hilbert function $\Delta H$
and create a set $Z$  of double
and simple points with $\Delta H = \Delta H_{Z}$.
Our answer to Question \ref{weakerq} is given
in Theorem \ref{lowerbounds} where we find a $d$,
that depends only on $\Delta H$, such that we can construct a
set of $d$ double points and $(\sum \Delta H) -3d$ reduced
 points whose Hilbert function is $H$.  In fact,
 for all $1 \leq d' \leq d$, we can find a
 scheme of $d'$ double points and $(\sum \Delta H) -3d'$
 reduced points with Hilbert function $H$ (see Corollary
 \ref{lowerboundcor}).  Moreover,
in Theorem \ref{HFofdoublepts} we give a condition
on a valid Hilbert function $H$ that guarantees that $H$ is
the Hilbert function of only double points.



We then  focus on the special cases that
\[\Delta H = (\underbrace{1,2,3,\ldots,t}_{t},\underbrace{t+1,\ldots,t+1}_{t})
~~\mbox{or}~~
\Delta H = (\underbrace{1,2,3,\ldots,t}_{t},\underbrace{t+1,\ldots,t+1}_{t},1)\]
In these cases, our construction produces $\binom{t+1}{2}$ double points,
respectively $\binom{t+1}{2}$ double points and one reduced points.   In
the first case,  the support of the points are the $\binom{t+1}{2}$ points of
intersection of $t$ general lines in $\mathbb{P}^2$.
This fact is equivalent to the statement that the points in the
support are a {\em star-configuration} of points in $\mathbb{P}^2$;
star configurations are widely studied, e.g. see \cite{CGvT2014,CvT2011}.
We prove (see Theorem \ref{doublestar}) that
this configuration is the
{\it only} configuration  of $\binom{t+1}{2}$ double points in $\mathbb{P}^2$
with $\Delta H_Z = \Delta H$.  In the second case, we show
(see Theorem \ref{puntosotto}) a similar
result by showing again that there is only one configuration of
$\binom{t+1}{2}$ double points and one reduced point that
has $\Delta H_Z = \Delta H$.

We conclude our paper with some final comments related
to how well our construction
performs, i.e., given a known valid Hilbert function
of $t$ double points, how many double points does
our procedure produce for the same valid Hilbert function.
In the case that the support of points is in generic position, we derive an asymptotic estimate.

 {\bf Acknowledgements.}  The computer algebra system
CoCoA \cite{C} played an integral role in this project.
The authors thank the hospitality of the Universit\`a di Catania
and McMaster University where
part of this work was carried out.
Carlini and Catalisano were
supported by GNSAGA of INDAM and by Miur (Italy) funds.
Guardo thanks FIR-UNICT
2014 and GNSAGA-INDAM for supporting part of the visit to McMaster
University.
Guardo's work has also
been supported by the Universit\`a degli Studi di
Catania, ``Piano della Ricerca 2016/2018 Linea di intervento 2".
Van Tuyl's research was supported by NSERC
Discovery Grant 2014-03898.


\section{Preliminaries}

We begin with a review of the relevant background; we continue to
use the notation and definitions given in the introduction.

\begin{definition}
A sequence  $\Delta H = (h_0,h_1,\ldots,h_{\sigma})$
  is a {\it valid Hilbert function of a set of points in $\mathbb{P}^2$}
  if there is an $0 < \alpha \leq \sigma$ such that
  \begin{enumerate}
  \item[$(a)$] $h_i = i+1$ if $0 \leq i < \alpha$, and
  \item[$(b)$] $h_i \geq h_{i+1}$ if $\alpha \leq i \leq \sigma$.
  \end{enumerate}
  Note that the indexing of $\Delta H$ begins with $0$.
\end{definition}

\begin{remark}\label{validhf}
It can be shown that $H: \mathbb{N} \rightarrow \mathbb{N}$ is a
Hilbert function of a set of points in $\mathbb{P}^2$ if and only
if $\Delta H (i)= H(i) - H(i-1)$ is a valid Hilbert
function.  More precisely, it was first shown by Geramita,
Maroscia, and Roberts \cite{GMR} that $H$ is the Hilbert function
of a reduced set of points in $\mathbb{P}^2$ if and only if
$\Delta H$ is the Hilbert function of an artinian quotient of
$k[x,y]$.  Using Macaulay's theorem which classifies all Hilbert
functions, one can determine all possible Hilbert functions of
artinian quotients of $k[x,y]$. In particular, one can show that
the Hilbert functions of artinian quotients of $k[x,y]$ must
satisfy the conditions of being a valid Hilbert function.   The
work of Geramita, Maroscia, and Roberts implies that if $Z
\subseteq \mathbb{P}^2$ is any set of fat points, then $\Delta
H_Z$ must satisfy the conditions $(a)$ and $(b)$ given above.
\end{remark}

\begin{definition}
Let $\Delta H = (1,2,\ldots,\alpha,h_{\alpha},\ldots,h_{\sigma})$ be a
valid Hilbert function of a set of points in $\mathbb{P}^2$.  We
define the {\it conjugate of} $\Delta H$, denoted $\Delta H^\star$, to the
be tuple
\[\Delta H^\star = (h_1^\star,\ldots,h_\alpha^\star) ~~
\mbox{where $h_i^\star = \#\{j ~|~ h_j \geq i\}$.}\]
\end{definition}

\begin{remark} The definition above is reminiscent of the conjugate
  of a partition.  Recall that
  a tuple $\lambda =(\lambda_1,\ldots,\lambda_r)$ of
   positive integers is a {\it partition} of an integer $s$ if $\sum_{i=1}^r
   \lambda_i = s$ and $\lambda_i \geq \lambda_{i+1}$ for every $i$.
   The {\it conjugate} \index{conjugate} of $\lambda$ is the tuple
   $\lambda_i^\star = \#\{j ~|~ \lambda_j \geq i\}$.   Note that $\Delta H$
   is not a partition, but $\Delta H^\star$ is a partition.  Furthermore,
   $h_1^\star = \sigma+1$ since there are $\sigma+1$ non-zero entries in
   $\Delta H$.
  \end{remark}

\begin{example} \label{runningex1}
Given a valid Hilbert function $\Delta H$, it is convenient to represent $\Delta H$ pictorially.  That is, we
make $\sigma+1$ columns of dots, where we place $h_i$ dots in the
$i$-th column.  For example, if $\Delta H = (1,2,3,4,4,3,1)$, then we
can represent $\Delta H$ pictorially as:
\[
\Delta H =
\begin{tabular}{cccccccc}
&         &           &         &$\bullet$&$\bullet$&          &         \\
&         &          &$\bullet$&$\bullet$ &$\bullet$&$\bullet$ &         \\
&         & $\bullet$&$\bullet$ &$\bullet$&$\bullet$&$\bullet$ &         \\
&$\bullet$&$\bullet$ &$\bullet$ &$\bullet$&$\bullet$&$\bullet$&$\bullet$ \\
\end{tabular}
\]
The tuple $\Delta H^\star = (7,5,4,2)$ can be read directly off of
this diagram; specifically, it is the number of dots in each row
reading from bottom to top.
\end{example}

Given a valid Hilbert function $\Delta H$, one can use
$\Delta H^\star$ to construct a set of reduced points $X \subseteq
\mathbb{P}^2$ such that  $H_X = H$ by building a suitable
$k$-configuration.   We present a specialization of this idea;  an
example appears as the first step of Example
\ref{illustrateex}.

\begin{theorem}\label{kconfig}
Let $\Delta H = (1,2,\ldots,\alpha,h_{\alpha},\ldots,h_{\sigma})$ be
a valid Hilbert function with
$\Delta H^\star = (h_1^\star,\ldots,h_\alpha^\star)$.
Let $\ell_1,\ldots,\ell_{\alpha}$ be $\alpha$ lines in $\mathbb{P}^2$ such
that no three lines meet at a point.   For $i = 1,\ldots,\alpha$,
let $X_i \subseteq \ell_i$ be any set of $h_i^\star$ points such that
$X_i \cap \ell_j = \emptyset$ for all $i \neq j$.   If
$X = X_1 \cup \cdots \cup X_\alpha$, then $H_X = H$.
\end{theorem}

\begin{proof} We sketch out the main ideas.   Our hypotheses
on the $X_i$'s implies that no point of $X$ is of the form
$\ell_i \cap \ell_j$ with $i \neq j$.  One can verify
that $h_1^\star > h_2^\star > \cdots > h_\alpha^\star$.  But
then $X$ is a $k$-configuration of type $(h_\alpha^\star,\ldots,h_1^\star)$
as first defined by
Roberts and Roitman \cite{RR} (and later generalized
by Geramita, Harima, and Shin \cite{GHS:1}).  In particular,
one can use \cite[Theorem 1.2]{RR} to compute the Hilbert function
of $X$ to show that it is the same as $H$.
\end{proof}

We now recall some crucial results from the work of
Cooper, Harbourne, Teitler \cite{CHT}.  We have
specialized their definitions to $\mathbb{P}^2$.

\begin{definition}[{\cite[Definition 1.2.5]{CHT}}] \label{susan}
Let $Z = m_1P_1 + m_2P_2 +\dots + m_{ s}P_{
s}$ be a fat point scheme in $\mathbb{P}^{2}$. Fix a
sequence $\ell_1,\dots ,\ell_{n}$ of lines in $\mathbb{P}^2$, not
necessarily distinct.
\begin{enumerate}
\item[(a)] Define the fat point schemes $Z_0,\dots ,Z_{n}$ by $Z_0 = Z$ and
$Z_j = Z_{j-1} : \ell_j$ for $1 \leq j \leq n$.  That is,
$Z_j$ is the scheme defined by $I_{Z_{j-1}}:\langle L_j \rangle$
if $L_j$ is the linear form defining $\ell_j$ and $I_{Z_j}$
is the ideal defining $Z_j$.
\item[(b)] The sequence $\ell_1,\dots ,\ell_{n}$
{\it totally reduces} $Z$ if $Z_{n} =\emptyset$ is the empty scheme.
This statement is equivalent to the property
that for each fat point $m_iP_i$, there are at least $m_i$ indices
$\{j_1,\ldots,j_{m_i}\}$ such that each $\ell_{j_k}$ passes through $P_i$.
\item[(c)] We associate with $Z$ and
the sequence $\ell_1, \dots ,\ell_{n}$ an integer vector
\[{\bf d} = {\bf d}(Z;\ell_1,\cdots ,\ell_{n}) =
(d_1,\dots , d_{n}),\]
where $d_j = \deg(\ell_j \cap Z_{j-1})$, the
degree of the scheme theoretic intersection of $\ell_j$ with
$Z_{j-1}$. We refer to ${\bf d}$ as the {\it reduction vector} for $Z$
induced by the sequence $\ell_1,\dots ,\ell_{n}$. We will say that
${\bf d}$ is a {\it full reduction vector} for $Z$ if $\ell_1,\dots
,\ell_{n}$ totally reduces $Z$.
\end{enumerate}
\end{definition}

\begin{remark}
If $Z$ is a fat point scheme, and if $P_{i_1},\ldots,P_{i_j}$ are
all the points in the support of $Z$ that lie on the line $\ell$,
then $\deg(\ell \cap Z) = m_{i_1}+\cdots + m_{i_j}$, i.e., the sum
of the multiplicities of the  points lying on  $\ell \cap Z$.  The
scheme $Z:\ell$ is the scheme that we  obtain by reducing the
multiplicities of $P_{i_1},\ldots,P_{i_j}$ by one (or removing the
point if its multiplicity is  1), and leaving the other
multiplicities alone.
\end{remark}

\begin{example} Consider three non-collinear points $P_1,P_2,P_3$
and the set of fat points $Z= 3P_1+3P_2+2P_3$.
Let $\ell_1=\ell_2$, $\ell_3$ and
$\ell_4$ be the lines through $P_1P_2$, $P_1P_3$,  and $P_2P_3$,
respectively. Then a full reduction vector for this scheme is
$(6,4,3,2)$. The pictures below show how to build this vector.
For example, in Figure \ref{firstfig}, the line $\ell_1$
passes through $P_1$ and $P_2$.  Since the multiplicity
of $P_1$ is three, and the same for $P_2$, we have $d_1 = 3+3 =6$.
We then reduce the multiplicity of $P_1$ and $P_2$ by one, as in
Figure \ref{secondfig}.  Then $d_2 = 2+2 =4$.
\begin{figure}[!htbp]
\centering
\mbox{%
\begin{minipage}{.50\textwidth}
\begin{tikzpicture}[scale=.9]
\clip(-0.3,-0.3) rectangle (5.5,3.5);
\draw [line width=1.2pt] (0.13809,1.)-- (4.7,1);
\draw (-0.4,1.5) node[anchor=north west] {$\ell_1$};
\begin{scriptsize}
\draw [fill=black] (1.,1.) circle (3pt);
\draw[color=black] (0.9,1.4) node {$3P_1$};
\draw [fill=black] (4.,1.) circle (3pt);
\draw[color=black] (4.3,1.4) node {$3P_2$};
\draw[color=black] (3,2.5) node {$2P_3$};
\draw [fill=black] (3.,3.) circle (3pt);
\end{scriptsize}
\end{tikzpicture}
\caption{}\label{firstfig}
\end{minipage}%
\hskip1cm
\begin{minipage}{0.50\textwidth}
\begin{tikzpicture}[scale=.9]
\clip(-0.3,-0.3) rectangle (5.5,3.5);
\draw [line width=1.2pt] (0.13809,1.)-- (4.7,1);
\draw (-0.4,1.5) node[anchor=north west] {$\ell_2$};
\begin{scriptsize}
\draw [fill=black] (1.,1.) circle (3pt);
\draw[color=black] (0.9,1.4) node {$2P_1$};
\draw [fill=black] (4.,1.) circle (3pt);
\draw[color=black] (4.3,1.4) node {$2P_2$};
\draw[color=black] (3,2.5) node {$2P_3$};
\draw [fill=black] (3.,3.) circle (3pt);
\end{scriptsize}
\end{tikzpicture}
\caption{}\label{secondfig}
\end{minipage}
}
\end{figure}
The line $\ell_3$
in Figure \ref{thirdfig}
passes through one point of multiplicity one and one of multiplicity
two, thus giving $d_3 = 3$.  Note that when we reduce
each multiplicity, the point $P_1$ is removed.  In the last step,
 we use the line $\ell_4$ to get $d_4 =2$ as in Figure \ref{fourthfig}.

\begin{figure}[!htbp]
\centering
\mbox{%
\begin{minipage}{.50\textwidth}
\begin{tikzpicture}[scale=.9]
\clip(-0.3,-0.3) rectangle (5.5,3.5);
\draw [line width=1.2pt] (0.24,0.24)-- (4.,4.);
\draw (-0.4,1.5) node[anchor=north west] {$\ell_3$};
\begin{scriptsize}
\draw [fill=black] (1.,1.) circle (3pt);
\draw[color=black] (0.9,1.4) node {$P_1$};
\draw [fill=black] (4.,1.) circle (3pt);
\draw[color=black] (4.3,1.4) node {$P_2$};
\draw[color=black] (3,2.5) node {$2P_3$};
\draw [fill=black] (3.,3.) circle (3pt);
\end{scriptsize}
\end{tikzpicture}
\caption{}\label{thirdfig}
\end{minipage}%
\hskip1cm
\begin{minipage}{0.50\textwidth}
\begin{tikzpicture}[scale=.9]
\clip(-0.3,-0.3) rectangle (5.5,3.5);
\draw [line width=1.2pt] (2.5,4.)-- (4.38,0.24);
\draw (4.6,1.5) node[anchor=north west] {$\ell_4$};
\begin{scriptsize}
\draw [fill=black] (4.,1.) circle (3pt);
\draw[color=black] (4.3,1.4) node {$P_2$};
\draw[color=black] (3,2.5) node {$P_3$};
\draw [fill=black] (3.,3.) circle (3pt);
\end{scriptsize}
\end{tikzpicture}
\caption{}\label{fourthfig}
\end{minipage}
}
\end{figure}

\end{example}

The next result is another specialization of Cooper,
Harbourne and Teitler \cite{CHT}.

\begin{theorem}\label{susan2}
Let $Z = Z_0$ be a fat point scheme in $\mathbb{P}^2$ with
full reduction vector ${\bf d} = (d_1,\dots , d_{n})$.
If $d_1 > d_2 > \cdots > d_{n}$,
then $H_Z$ only depends on ${\bf d}$.
\end{theorem}

\begin{proof}
From our assumption, we have $d_{i}-d_{i+1}\geq 1$ for all
$i=1,\dots,n$. Then  $
d_i-d_{i+p}\geq p$ for all $i$.
This implies that ${\bf d} = (d_1,\dots , d_{n})$ is a GMS vector (see
Definition 2.2.1 in \cite{CHT}).  Then Theorem 2.2.2
and Section 2.3 of \cite{CHT} imply
$$H_{Z}(t)=\sum_{i=0}^{n-1} \binom{\min\{t-i+1,d_{i+1}\}}{1},$$
that is, $H_Z$ can be computed directly from ${\bf d}$.
\end{proof}


\section{An operation that preserves the Hilbert function} \label {construction}

In this section, we first show that under certain conditions, we
can degenerate a fat point scheme $Z$ consisting of double points and
reduced points to make a new fat point scheme $Z'$ consisting of
one additional double point, and three less reduced
points. Furthermore, the two schemes will have the same Hilbert
function. Note that degeneration techniques have been successfully used in other situations, e.g. see \cite{sundials,marvialessandro}.

By repeatedly applying this procedure, we can do the following.
Let $\Delta H$ be a valid Hilbert function of a set of points.
 Theorem \ref{kconfig} implies that there
is a set of reduced points $X$ with Hilbert function
$\Delta H$ that satisfies the hypotheses of our procedure given
below.
We can then remove three points from $X$ and add a double point
to make a set $Z$ of fat points  with the same Hilbert function as
$\Delta H$.  We can continue this procedure (provided the hypotheses
of our procedure are still satisfied) to build sets of fat points
consisting of double and reduced points that have Hilbert function
$\Delta H$.  This procedure will then allow us to
give an answer to Question \ref{weakerq}.

We now state and prove the main step in our procedure.

\begin{theorem}\label{procedure}
Let $\ell_1,\ldots,\ell_{n}$ be $n$ lines in $\mathbb{P}^2$ such
that no three lines meet at a point.   Let $P_{i,j} = \ell_i \cap
\ell_j$ for $1 \leq i < j \leq {n}$.  Suppose that $Z$ is a set of
double points and reduced points in $\mathbb{P}^2$ that satisfies
the following conditions:
\begin{enumerate}
\item[$(a)$] ${\rm Supp}(Z) \subseteq \bigcup_{i=1}^n \ell_i$, i.e.,
all the points in the support lie on the lines $\ell_i$.
\item[$(b)$] If $2P$ is a double point of $Z$, then $P = P_{i,j}$
for some $i < j$, i.e., all double points of $Z$ lie at an
intersection point of two $\ell_p$'s.
\item[$(c)$] If $Q$ is a reduced point of $Z$ and
$Q \in \ell_i$, then $Q \neq \ell_i \cap \ell_p$ for $p\neq i$,
i.e., the reduced points do not lie at an intersection point.
\item[$(d)$] If
\[
d_j = \deg(Z_{j-1} \cap \ell_j)
\]
for $j=1,\ldots,n$, then $d_1 > d_2 > \cdots > d_n$, where $Z_j=Z_{j-1}:\ell_{i}$
and we set $Z_0=Z$.

\item[$(e)$] There exist $i,j$ with $i < j$ such that $Z$ contains two
reduced points $Q_1,Q_2 \in \ell_i$ and a reduced point $R \in \ell_j$,
but $2P_{i,j}$ is not a double point of $Z$.
\end{enumerate}
Let $Z'$ be the set of double and reduced points obtained by
adding the double point $2P_{i,j}$ to $Z$, and removing
the reduced points $\{R,Q_1,Q_2\}$.  Then $Z$ and $Z'$ have
the same Hilbert function.
\end{theorem}

\begin{proof}
We begin by observing that ${\rm Supp}(Z')$ is also contained in
$\bigcup_{i=1}^n \ell_i$ by our construction since the only point
we added to the support is $P_{i,j}$.  This observation and $(a)$
imply that the lines $\ell_1,\ldots,\ell_n$ totally reduce both
$Z$ and $Z'$.  Indeed, if $Q$ is a reduced point of $Z$,
respectively $Z'$, then it lies on some distinct $\ell_i$ by
$(c)$. If $2P$ is a double point of $Z$, or $Z'$, then $P = \ell_i
\cap \ell_j$ for some $i < j$ by $(b)$ (or by the construction of
$Z'$), so there are at least two $\ell_p$'s that pass through
$2P$.   It follows from the equivalent statement in Definition
\ref{susan} that the $\ell_i$'s totally reduce $Z$ and $Z'$.

To finish the proof, we claim it is enough to show that
\[ \deg(Z_{j-1} \cap \ell_j) =
\deg(Z'_{j-1} \cap \ell_j) ~~\mbox{for all $1 \leq j \leq
n$.}\]
Indeed, if this fact is true, then part $(d)$ and Theorem
\ref{susan2} imply that the Hilbert function of $Z$ and $Z'$ are the same.

To verify the claim, we first observe that our change from $Z$ to
$Z'$ only effects the points on the lines $\ell_i$ and $\ell_j$,
and consequently, could only effect the value of $\deg(Z'_{p-1} \cap \ell_p)$ for $p =i$ and $j$. In  the
computation of $\deg(Z_{i-1} \cap \ell_i)$
we get a contribution of two from each reduced point $Q_1$ and
$Q_2$.  Those two points do not contribute to  $\deg((Z'_{i-1} \cap \ell_i)$ since we have removed them, but
the fat point $2P_{i,j}$ (which is not in $Z$) contributes two to
the degree.  The other points of $Z_{i-1}$ and
$Z'_{i-1}$ on $\ell_i$ remain the same, so
they contribute equally to the degree.  So $\deg(Z_{i-1} \cap \ell_i) = \deg(Z'_{i-1}\cap \ell_i).$

When we compute $\deg(Z_{j-1} \cap \ell_j)$
we get a contribution of one from $R$.  This point does not
contribute to  $\deg(Z'_{j-1} \cap \ell_j)$
since it was removed.  We, however, get a
contribution of one from $P_{i,j}$.   (The multiplicity of
$P_{i,j}$ was dropped from two to one when we formed
$Z'_{i}$.)  As we mentioned above, the other
points on $\ell_j$ contribute the same. So, again we have
$\deg(Z_{j-1} \cap \ell_j) = \deg(Z'_{j-1} \cap \ell_j)$.  This completes the proof.
\end{proof}

\begin{remark}
  We note that the hypothesis of Theorem \ref{procedure} are sufficient conditions
  which allow us, for example, to apply the results  of \cite{CHT}.
  Consider condition $(d)$ on the degrees $d_i$. If some of the inequalities
  do not hold, then the conclusion might be
  false. Let $Z$ be a set of five points supported on the union of the
  lines $\ell$ and $\ell'$, namely two points on the former and
  three on the latter. Thus $\Delta H_Z=(1,2,2)$. If we set $\ell_1=\ell$
  and $\ell_2=\ell'$, we get $d_1=2$ and $d_2=3$, and condition $(c)$ is not
  satisfied. Indeed, the resulting set $Z'$ is such that
  $
  \Delta H_{Z'}=(1,2,1,1)$.
  Hence the two Hilbert functions are not equal.
\end{remark}

\begin{remark}
There are also counterexamples when the degrees $d_i$ are not all distinct.
  Consider, for example, the complete intersection of a cubic
  $\ell_1\cup\ell_2\cup\ell_3$ with the union of two distinct lines.
  This set of six points lie on a conic, but applying our construction
  leads to a scheme of one double point and three simple points not
  lying on a conic. That is, $d_1=d_2=d_3=2$, and our construction
  does not preserve the Hilbert function in degree two.
\end{remark}

\begin{remark}
When we construct $Z'$, then $Z'$ will also satisfy the
hypotheses $(a)-(d)$ of Theorem \ref{procedure}.  If
$Z'$ also satisfies $(e)$, then we can add another double point, and
so on, until hypothesis $(e)$ is no longer satisfied.
\end{remark}

We expand upon the above remark.
Given a valid Hilbert function $\Delta H$ with $\Delta H^{\star}
= (h_1^\star,\ldots,h_\alpha^\star)$,
we present a construction based upon
Theorem \ref{procedure} which will allow us to produce
a set $Z$ of double and simple points such that
$\Delta H=\Delta H_Z$.   The rough idea behind our construction is
to start with a set of reduced points with the correct Hilbert function,
and then, in a controlled fashion, repeatedly replace three reduced points
with a double point, and use Theorem \ref{procedure} to show that
the Hilbert function does not change after each iteration.
Below, we will use the notation $n_+ = \max\{n,0\}$.

\begin{construction}\label{algorithm}\hspace{.1cm}
\vspace{.1cm}

\noindent
\begin{tabular}{rl}
{\rm {\bf INPUT:}}& A valid Hilbert function $\Delta H$ with
$\Delta H^\star = (h_1^\star,\ldots,h_\alpha^\star)$.\\
{\rm {\bf OUTPUT:}}& A scheme $Z$ of double points and reduced points
in $\mathbb{P}^2$ with $H_Z = H$.
\end{tabular}
\vspace{.25cm}

\noindent {\rm STEP 0.}
Let $\ell_1,\ldots,\ell_{\alpha}$ and $P_{i,j}$ be as
in Theorem \ref{procedure}.
Let $Z_0$ be a set reduced points of $\mathbb{P}^2$
with $H_{Z_0} = H$ as constructed in Theorem \ref{kconfig} with
$Z_0 \subseteq \bigcup_{i=1}^\alpha \ell_i$
such that $|Z_0 \cap \ell_i| = h_i^\star$.   Continue to {\rm STEP $1$}.


\noindent For $n \geq 1$:

\noindent {\rm STEP $n$}.
Set ${\bf h_n} = ((h_n^\star-(n-1))_+,\ldots,(h_\alpha^\star-(n-1))_+)$
and
\[s_n = \#\{k ~|~ n+1 \leq k \leq \alpha ~\mbox{and}~ h^\star_k \geq n \}.\]
If $s_n = 0$, then return $Z_{n-1}$.  Otherwise, let
\[t_n =  \min\left\{\left\lfloor \frac{(h_n^\star-(n-1))_+}{2}\right\rfloor, s_n \right\}.\]
Remove $2t_n$ points on $\ell_n$
and one point on each $\ell_j$ for $j=n+1,\dots,n+t_n$ from $Z_{n-1}$, and
add the double points  $2P_{n,j}$ where $P_{n,j} = \ell_n\cap \ell_j$, for
$j=n+1,\dots,n+t_n$.  Let $Z_{n}$ denote the resulting scheme.
Continue to {\rm STEP $n+1$}.
\end{construction}

\begin{proof} We verify that Construction \ref{algorithm}
produces the stated output.  First, note that the construction will stop
at (or before) STEP $\alpha$ because $s_\alpha=0$.

We now show that the scheme $Z_n$ of double and reduced points
construced at STEP $n$ has the property that $H_{Z_n} = H$.
We proceed by induction on $n$.  If $n=0$, then the set of
reduced points $Z_0$ has the property $H_{Z_0} = H$ by Theorem \ref{kconfig}.

So, now assume that $n \geq 1$, and that $Z_{n-1}$ satisfies the induction
hypothesis. We first observe that the value $(h_i^\star-(n-1))_+$ in the
tuple
${\bf h_n} =
((h_n^\star-(n-1))_+,\ldots,(h_\alpha^\star-(n-1))_+)$ counts the minimal possible
number of reduced points on $\ell_i$ for $i=n,\ldots,\alpha$ after
constructing $Z_{n-1}$.   This is because in each of STEP $1$ through $n-1$,
we used at most one reduced point on $\ell_i$ when constructing $Z_j$
with $0 \leq j \leq n-1$.

 If $s_n =0$, then our procedure terminates with
$Z_{n-1}$, which, by induction, is a set of double and reduced points
with $H_{Z_{n-1}}=H$.  Now suppose that $s_n > 0$, i.e., there exists
a $n+1 \leq k \leq \alpha$ such that $h_k^\star \geq n$, or equivalently,
$h_k^\star-(n-1) \geq 1$.   In fact, because
$h_1^\star > h_2^\star > \cdots > h_n^{\star} > \cdots > h_{k}^\star > \cdots > h_\alpha^\star$, we can assume $h_{n+j}^\star - (n-1) \geq 1$ for $j=1,\dots,s_n$,
and also $h_n^\star-(n-1) \geq 2.$  Consequently,
\[t_n =  \min\left\{\left\lfloor \frac{(h_n^\star-(n-1))_+}{2}\right\rfloor, s_n \right\} \geq 1.\]
From our description of the entries of ${\bf h_n}$, it follows that
we can find $2t_n$ reduced points on $\ell_n$ and a reduced point
on each of the lines $\ell_{n+1},\ldots,\ell_{n+t_n}$ in $Z_{n-1}$.
For each $j = 1,\ldots,t_n$, we remove two reduced points from $\ell_n$ and
one reduced point from $\ell_{n+j}$, and replace these three reduced
points  with the double $2P_{n,n+j}$ where $P_{n,n+j} = \ell_n \cap \ell_{n+j}$.
By repeatedly applying Theorem \ref{procedure}, each time we add
a new double point and remove the reduced points, the new scheme has
the same Hilbert function.    Consequently,
the scheme $Z_{n}$ produced by STEP $n$ satisfies
$H_{Z_n} = H_{Z_{n-1}} = H$, as desired.
\end{proof}

\begin{example}\label{illustrateex}
We illustrate Construction \ref{algorithm}
with the valid Hilbert function
$\Delta H = (1,2,3,4,2)$.  In this case $\Delta H^\star = (5,4,2,1)$.   Fix
four general lines $\ell_1,\ell_2,\ell_3,$ and $\ell_4$, i.e., no
three of the lines meet at a point.
If we
place five points on $\ell_1$, four points on $\ell_2$, two points
on $\ell_3$, and one point on $\ell_4$, as in the  Figure
\ref{step1}, then the set of reduced points $Z_0$ has Hilbert
function $\Delta H_X = \Delta H$  (by Theorem \ref{kconfig}).  The
construction of $Z_0$ is STEP $0$ of Construction \ref{algorithm}.

For STEP $1$, we let ${\bf h_1}= (5,4,2,1)$ and $s_1 = 3$.  Since $s_1 \neq 0$,
we let $t_1 = \min\{\lfloor \frac{5}{2} \rfloor,3\} = 2$.
We remove $2\cdot 2 = 4$ points from $\ell_1$, and $1$ point from $\ell_2$
and $1$ point from $\ell_3$, and we add the double points $2P_{1,2}$ and
$2P_{1,3}$ to make the scheme $Z_1$ as in Figure \ref{step2}.  The double
points are denoted with a $2$ in the figure.   Note that
by Theorem \ref{procedure} this scheme has the same Hilbert function as
$Z_0$. Roughly speaking, we are ``merging'' two points on $\ell_1$ with
a third point on $\ell_2$ (or $\ell_3$) to make the
double point $2P_{1,2}$ (or $2P_{1,3}$).

\begin{figure}[!htbp] \label{fig1}
\centering
\mbox{%
\begin{minipage}{0.50\textwidth}
\begin{tikzpicture}[scale=0.6]
\clip(-0.7,-2.5) rectangle (10,4);
\draw [line width=0.8pt] (0.43356525262598167,-0.837313686104764)-- (8.854928961644573,1.7840629353535518);
\draw [line width=0.8pt] (2.435895714524843,-1.8956076501016481)-- (8.113055412043682,3.8363202135355516);
\draw [line width=0.8pt] (6.112048993437138,-2.45813447524598)-- (6.362749237717014,3.943078428700201);
\draw [line width=0.8pt] (0.4876183843762104,0.48042166509819506)-- (8.859487147904208,-1.6315379685322804);
\draw (0.5792786199095276,0.44057751847052556) node[anchor=north west] {$\ell_1$};
\draw (0.8937095741126032,-0.5689113344972426) node[anchor=north west] {$\ell_2$};
\draw (2.99543226799632,-1.429459209158291) node[anchor=north west] {$\ell_3$};
\draw (6.288682788333798,-1.5949491850546464) node[anchor=north west] {$\ell_4$};
\begin{scriptsize}
\draw [color=black] (7.401831311192663,-1.2638171676195675) circle (3.5pt);
\draw [color=black] (6.961471469000992,-1.152728205614772) circle (3.5pt);
\draw [color=black] (7.788999001875781,-1.36148741989716) circle (3.5pt);
\draw [color=black] (8.228222220205335,-1.4722896473886444) circle (3.5pt);
\draw [color=black] (8.563585253254768,-1.556891211079881) circle (3.5pt);
\draw [fill=black] (7.294492047969021,1.2983348530491288) circle (3.5pt);
\draw [fill=black] (7.658795156625736,1.4117340227269093) circle (3.5pt);
\draw [fill=black] (8.027270583061497,1.5264319388620566) circle (3.5pt);
\draw [fill=black] (8.404090645055357,1.6437273479119399) circle (3.5pt);
\draw [color=black] (7.197494685265647,2.911926974612835) ++(-3.5pt,0 pt) -- ++(3.5pt,3.5pt)--++(3.5pt,-3.5pt)--++(-3.5pt,-3.5pt)--++(-3.5pt,3.5pt);
\draw [color=black] (7.468216369200496,3.185260339555014) ++(-3.5pt,0 pt) -- ++(3.5pt,3.5pt)--++(3.5pt,-3.5pt)--++(-3.5pt,-3.5pt)--++(-3.5pt,3.5pt);
\draw [fill=black] (6.336484434722719,3.2724504589125343) ++(-3.5pt,0 pt) -- ++(3.5pt,3.5pt)--++(3.5pt,-3.5pt)--++(-3.5pt,-3.5pt)--++(-3.5pt,3.5pt);
\end{scriptsize}
\end{tikzpicture}
\caption{The set $Z_0$ with $H_{Z_0} = H$} \label{step1}
\end{minipage}%
\hskip5pt
\begin{minipage}{0.50\textwidth}
\begin{tikzpicture}[scale=0.6]
\clip(-0.7,-2.5) rectangle (10,4);
\draw [line width=0.8pt] (0.43356525262598167,-0.837313686104764)-- (8.854928961644573,1.7840629353535518);
\draw [line width=0.8pt] (2.435895714524843,-1.8956076501016481)-- (8.113055412043682,3.8363202135355516);
\draw [line width=0.8pt] (6.112048993437138,-2.45813447524598)-- (6.362749237717014,3.943078428700201);
\draw [line width=0.8pt] (0.4876183843762104,0.48042166509819506)-- (8.859487147904208,-1.6315379685322804);
\draw (0.5743549512051723,0.4195151579018995) node[anchor=north west] {$\ell_1$};
\draw (0.8978126313662312,-0.5949657480577836) node[anchor=north west] {$\ell_2$};
\draw (3.0002875524131136,-1.4477178139369373) node[anchor=north west] {$\ell_3$};
\draw (6.293674841325713,-1.6094466540174666) node[anchor=north west] {$\ell_4$};
\begin{scriptsize}
\draw [fill=black] (2.8,-0.09) circle (3.5pt);
\draw[color=black] (2.9,.3) node {$2$};
\draw [fill=black] (3.92,-.37) circle (3.5pt);
\draw [color=black] (3.95,-.07) node {$2$};
\draw [color=black] (8.563585253254768,-1.556891211079881) circle (3.5pt);
\draw [fill=black] (7.658795156625736,1.4117340227269093) circle (3.5pt);
\draw [fill=black] (8.027270583061497,1.5264319388620566) circle (3.5pt);
\draw [fill=black] (8.404090645055357,1.6437273479119399) circle (3.5pt);
\draw [color=black] (7.468216369200496,3.185260339555014) ++(-3.5pt,0 pt) -- ++(3.5pt,3.5pt)--++(3.5pt,-3.5pt)--++(-3.5pt,-3.5pt)--++(-3.5pt,3.5pt);
\draw [fill=black] (6.336484434722719,3.2724504589125343) ++(-3.5pt,0 pt) -- ++(3.5pt,3.5pt)--++(3.5pt,-3.5pt)--++(-3.5pt,-3.5pt)--++(-3.5pt,3.5pt);
\end{scriptsize}
\end{tikzpicture}
\caption{The scheme $Z_1$ with $H_{Z_1} = H$}   \label{step2}
\end{minipage}
}
\end{figure}

Moving to STEP  $2$, we set ${\bf h_2} = ((4-1)_+,(2-1)_+,(1-1)_+) =(3,1,0)$
and $s_2 = 1$.  (Note the $j$-th entry of ${\bf h_2}$ is a lower bound on
the number of reduced points on $\ell_{j+1}$ for $j=2,3,4$.)
Because $s_2 \neq 0$, we let
$t_2 = \min\{\lfloor \frac{3}{2} \rfloor, 1\} = 1$.   We remove two points
from $\ell_2$ and one point from $\ell_3$ from $Z_1$,
but add the double point $2P_{2,3}$ to form the scheme $Z_2$.
Again, Theorem \ref{procedure} implies that this construction
of double and reduced points has the same Hilbert function as $Z_1$,
and consequently,
$Z$.  See Figure \ref{step3} for an illustration of $Z_2$.


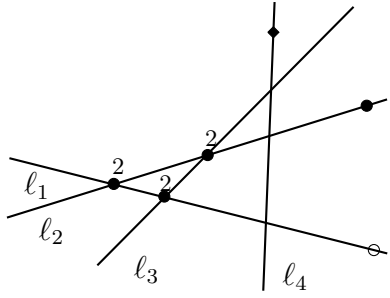
\begin{figure}[!htbp]
\centering
\mbox{%
\begin{minipage}{0.50\textwidth}
\begin{tikzpicture}[scale=0.6]
\clip(-0.7,-2.5) rectangle (10,4);
\draw [line width=0.8pt] (0.43356525262598167,-0.837313686104764)-- (8.854928961644573,1.7840629353535518);
\draw [line width=0.8pt] (2.435895714524843,-1.8956076501016481)-- (8.113055412043682,3.8363202135355516);
\draw [line width=0.8pt] (6.112048993437138,-2.45813447524598)-- (6.362749237717014,3.943078428700201);
\draw [line width=0.8pt] (0.4876183843762104,0.48042166509819506)-- (8.859487147904208,-1.6315379685322804);
\draw (0.5743549512051723,0.4195151579018995) node[anchor=north west] {$\ell_1$};
\draw (0.8978126313662312,-0.5949657480577836) node[anchor=north west] {$\ell_2$};
\draw (3.0002875524131136,-1.4477178139369373) node[anchor=north west] {$\ell_3$};
\draw (6.293674841325713,-1.6094466540174666) node[anchor=north west] {$\ell_4$};
\begin{scriptsize}
\draw [fill=black] (2.8,-0.09) circle (3.5pt);
\draw[color=black] (2.9,.3) node {$2$};
\draw [fill=black] (3.92,-.37) circle (3.5pt);
\draw [color=black] (3.95,-.07) node {$2$};
\draw [fill=black] (4.88,0.55) circle (3.5pt);
\draw [color=black] (4.96,0.92) node {$2$};
\draw [color=black] (8.563585253254768,-1.556891211079881) circle (3.5pt);
\draw [fill=black] (8.404090645055357,1.6437273479119399) circle (3.5pt);
\draw [fill=black] (6.336484434722719,3.2724504589125343) ++(-3.5pt,0 pt) -- ++(3.5pt,3.5pt)--++(3.5pt,-3.5pt)--++(-3.5pt,-3.5pt)--++(-3.5pt,3.5pt);
\end{scriptsize}
\end{tikzpicture}
\caption{The scheme $Z_2$ with $H_{Z_2} = H$} \label{step3}
\end{minipage}
}
\end{figure}
Finally, at STEP $3$ we have ${\bf h_3} = ((2-2)_+,(1-2)_+) = (0,0)$ and
$s_3 =0$. Our construction terminates and returns the scheme $Z_2$, which
consists of three double points and three reduced
points, and $\Delta H_{Z_2} = \Delta H = (1,2,3,4,2)$.
\end{example}

\begin{example}
Construction \ref{algorithm} may not produce a set of just double points,
even if $\Delta H$ is known to be the valid Hilbert function
of a set of double points.  For example,
consider the scheme $Z$ of two double points. We have $\Delta H_Z=(1,
2, 2, 1)$.  So, $\Delta H = (1,2,2,1)$ is a valid Hilbert function of a
set of double points.  However, Construction \ref{algorithm}
cannot be used to detect this fact.
In particular,  since $\Delta H^\star =(4,2)$
Construction \ref{algorithm} returns only one double
point (and two simple points on $\ell_1$, and an other simple point
on $\ell_2$).
\end{example}

\begin{example}
Consider the valid Hilbert function
$\Delta H = (\underbrace{1,\ldots,1}_{\sigma+1})$.
Construction \ref{algorithm} applied to this $\Delta H$
stops at the beginning of
STEP $1$ by producing $\sigma+1$
reduced points on a line $\ell_1$.  A valid Hilbert function of this type is the
only time Construction \ref{algorithm}
terminates at STEP 1.  Note that it can be shown that if
 $\Delta H = (1,\ldots,1)$, then the only zero-dimensional scheme $Z$
with $\Delta H_Z = \Delta H$ is precisely a set of reduced points on a line.
If $\Delta H = (1,2,\ldots)$, then Construction \ref{algorithm} will produce
a scheme $Z$ with at least one double point.
\end{example}

As noted in the last remark, if $\Delta H = (1,2,\ldots)$, then our procedure
produces a scheme with at least one double point.
We are actually interested in producing a scheme
with the largest possible number of double points.  We can determine
the number of double points produced by Construction \ref{algorithm}
directly from $\Delta H$, as shown in the next result
which gives an answer to Question \ref{weakerq}.

\begin{theorem}\label{lowerbounds}
Let $\Delta H$ be a valid Hilbert function with $\Delta H^{\star}
= (h_1^\star,\ldots,h_\alpha^\star)$.
Set
\[d = \sum_{i=1}^{\alpha-1} \min\left\{
\left\lfloor \frac{(h_i^\star-(i-1))_+}{2} \right\rfloor,
\#\{k ~|~ i+1 \leq k \leq \alpha ~\mbox{and}~~ h^\star_k \geq i \}
\right\}.
\]
 Then there is a set
of fat points $Z$ of $d$ double points and $r = \left(\sum \Delta H \right) - 3d$
reduced points with $\Delta H_Z = \Delta H$.
\end{theorem}

\begin{proof}
We first note that $s_i$ in Construction \ref{algorithm}
equals $\#\{k ~|~ i+1 \leq k \leq \alpha ~\mbox{and}~~ h^\star_k \geq i \}$
Also, we have $s_1 \geq s_2 \geq \cdots \geq s_{\alpha-1} \geq s_{\alpha} = 0$.
Let $j = \max\{i ~|~ s_i \neq 0\}$.  So, for $i=1,\ldots,j$,
\[ \min\left\{
\left\lfloor \frac{(h_i^\star-(i-1))_+}{2} \right\rfloor,
\#\{k ~|~ i+1 \leq k \leq \alpha ~\mbox{and}~~ h^\star_k \geq i \}
\right\}.\]
is the number of double points added to $Z_{i-1}$ at STEP $i$.
For $i=j+1,\ldots,\alpha$,  Construction \ref{algorithm} has
terminated, and no new double points are added.  This
fact is captured in the summation via the fact that
all the summands
\[\min\left\{
\left\lfloor \frac{(h_i^\star-(i-1))_+}{2} \right\rfloor,
\#\{k ~|~ i+1 \leq k \leq \alpha ~\mbox{and}~~ h^\star_k \geq i \}
\right\}=0\]
for $i=j+1,\ldots,\alpha$.
\end{proof}

We record some corollaries:

\begin{corollary}\label{lowerboundcor}Let $\Delta H = (1,2,\ldots,\alpha,h_{\alpha},\ldots,h_{\sigma})$
be a valid Hilbert function, and let $d$ be as in Theorem \ref{lowerbounds}.
\begin{enumerate}
\item[$(i)$] The integer $d$ satisfies
$\min\left\{\left\lfloor\frac{\sigma+1}{2}\right\rfloor,
\alpha-1 \right\}\leq d \leq \binom{\alpha}{2}$.
\item[$(ii)$] For every integer $1 \leq e \leq d$, there exists a scheme $Z$
of $e$ double points and $(\sum \Delta H) - 3e$ reduced points with
$H_Z = H$.
\item[$(iii)$]  Let $e = \min\left\{\lfloor \frac{\sigma+1}{2} \rfloor,
\alpha-1 \right\}$.  Then there exists a scheme $Z$ with $e$ double
points and  $(\sum \Delta H) - 3e$ reduced points.
\end{enumerate}
\end{corollary}

\begin{proof}
For $(i)$, the support of each double point produced by Construction \ref{algorithm} has the form $\ell_i \cap \ell_j$.
Since there are only $\alpha$ lines $\ell_1,\ldots,\ell_\alpha$
used in this construction, the outputted scheme can
have at most $\binom{\alpha}{2}$ double points, thus
giving the upper bound.  For the lower bound, note
that in the computation of $d$ using Theorem \ref{lowerbounds},
when the index is $i=1$, we have  $h_1^\star = \sigma+1$ and
$\alpha-1 = \#\{2 \leq k \leq \alpha ~\mbox{and}~ h_k^\star \geq 1\}$, i.e.,   the lower bound is the first term
in the sum.

The proof of $(iii)$ will follow from $(i)$ and $(ii)$
since $e \leq d$.
For $(ii)$, notice that Construction \ref{algorithm} adds one double point at a time,
and when it finishes, we have $d$ double points.  Since $e \leq d$,
we use this procedure again, but changing our stopping criterion
so the procedure terminates when we have $e$ double points.
\end{proof}

Recall that one of the fundamental problems about Hilbert functions of double
points in $\mathbb{P}^2$ is to classify what functions are the Hilbert
functions of double (or more generally, fat) points.  We can now
contribute to this problem by identifying some new functions as the
Hilbert functions of double points.

\begin{theorem}\label{HFofdoublepts}
Let $\Delta H$ be a valid Hilbert function, and let $\Delta H^\star =
(h_1^\star,\ldots,h_\alpha^\star)$.
Suppose that
\[\frac{(h_i^\star-(i-1))_+}{2} =
\#\{k ~|~ i+1 \leq k \leq \alpha ~\mbox{and~ $h_k^\star \geq i$}\}
~~\mbox{for $i = 1,\ldots \alpha-1$}.\]
Then $\Delta H$ is Hilbert function of a set of double points in
$\mathbb{P}^2$.
\end{theorem}

\begin{proof}
We prove this by showing that Construction \ref{algorithm} terminates with
a scheme of only double points.
At STEP 1, observe that the vector
${\bf h}_1 = (h_1^\star,\ldots,h_\alpha^\star)$ has the property that
$h_i^\star$ is the exact number of reduced points on $\ell_i$.
The hypotheses imply that when Construction \ref{algorithm} executes STEP $1$,
the value of $t_1$ is given by
\[t_1 =   \frac{(h_1^\star)_+}{2} = s_1.\]
But this means that all the reduced points on $\ell_1$ and one
reduced point on each of $\ell_{1+1},\ldots,\ell_{1+s_1 = \alpha}$ are removed to
form $2t_1$ double points.  In particular, no reduced points are left on $\ell_1$
and  ${\bf h}_2 = ((h_2^\star-1)_+,\ldots,(h_\alpha^\star-1)_+)$   has the property that
$(h_i-1)_+^\star$ is exactly the number of reduced points remaining on $\ell_i$.
(In the general, case, $(h_i^\star-1)_+$ is only a lower bound for the number
of reduced points that remain).

Proceeding by induction on $n$, note that at STEP $n$, the tuple  ${\bf h_n} =
((h_n^\star-(n-1))_+,\ldots,(h_\alpha^\star-(n-1))_+)$ counts exactly the
number of reduced points on $\ell_i$ for $i=n,\ldots,\alpha$ after
constructing $Z_{n-1}$.   This is because in each of STEP $1$ through $n-1$,
we used at exactly one reduced point on $\ell_i$ when constructing $Z_j$
with $0 \leq j \leq n-1$.
The hypotheses imply when Construction \ref{algorithm} executes STEP $n$ we
get
\[t_n =  \frac{(h_n^\star-(n-1))_+}{2}  = s_n.\]
So,  all the remaining reduced points on $\ell_n$ and one
reduced point on each of $\ell_{n+1},\ldots,\ell_{n+s_n}$ are removed to
form $2t_n$ double points.  In particular, no reduced points are left on $\ell_n$
and  ${\bf h}_{n+1} = ((h_{n+1}^\star-n)_+,\ldots,(h_\alpha^\star-(n))_+)$
has the property that $(h_i-n)_+^\star$ is the exactly the number of
reduced points remaining on $\ell_i$.

When the algorithm terminates, ${\bf h_n} = (0,\ldots,0)$, that is, no reduced
points remain.
\end{proof}

\begin{example}
Consider the valid Hilbert function
\[\Delta H = (1,2,3,4,5,6,2,2,1,1) \Leftrightarrow \Delta H^\star= (10,7,4,3,2,1).\]
Then $\Delta H^\star$ statisfies the hypothesis of Theorem \ref{HFofdoublepts}
since
\begin{eqnarray*}
\frac{h_1^\star}{2} & = & 5 =  \#\{k ~|~ 2 \leq k \leq 6 ~\mbox{and~ $h_k^\star
\geq 1$}\}\\
\frac{(h_2^\star-1)_+}{2} & = & 3 =  \#\{k ~|~ 3 \leq k \leq 6 ~\mbox{and~ $h_k^\star \geq 2$}\}\\
\frac{(h_3^\star-2)_+}{2} & = & 1 =  \#\{k ~|~ 4 \leq k \leq 6 ~\mbox{and~ $h_k^\star \geq 3$}\}\\
\frac{(h_i^\star-(i-1))_+}{2} & = & 0 =  \#\{k ~|~ i+1 \leq k \leq 6 ~\mbox{and~ $h_k^\star \geq i$}\} ~~\mbox{for $i \geq 4$.}\\
\end{eqnarray*}
Then $\Delta H$ is the Hilbert function of $(\sum \Delta H)/3 = 9$ double points.
Indeed, if $\ell_1,\ldots,\ell_6$ are six general lines such that no three meet
at a point, Construction \ref{algorithm} will produce the scheme of nine
double points
\[Z = 2P_{1,2} + 2P_{1,3} + 2P_{1,4} + 2P_{1,5} + 2P_{1,6} + 2P_{2,3} + 2P_{2,4} +
2P_{2,5} + 2P_{3,4}\]
where $P_{i,j} = \ell_i \cap \ell_j$ with Hilbert function $H_Z = H$.
\end{example}

The following corollary is used in the next section.
\begin{corollary}\label{onlydoubleptsspecial}
Let $\Delta H$ be a valid Hilbert function with $\Delta
H^\star = (h_1^\star,\ldots,h_\alpha^\star)$. If
 \[\frac{h_i^\star-(i-1)}{2} = \alpha- i~~\mbox{for $i = 1,\ldots \alpha$},\]
then there exists a scheme $Z$
in $\mathbb{P}^2$ of only double points with $\Delta H=\Delta H_Z$.
\end{corollary}
\begin{proof} We have
\[\Delta H^\star=(\underbrace{2\alpha -2,\dots,\alpha+1,\alpha,\alpha-1}_{\alpha}).\]
Now apply Theorem \ref{HFofdoublepts}.
\end{proof}


\section{Special configurations}\label{specialconfigsection}

In this section, we will examine two valid Hilbert functions:
\[\Delta H_1 = (\underbrace{1,2,3,\ldots,t}_{t},\underbrace{t+1,\ldots,t+1}_{t}),
~~\mbox{and}~~~
\Delta H_2 = (\underbrace{1,2,3,\ldots,t}_{t},\underbrace{t+1,\ldots,t+1}_{t},1).\]
In the first case, we will show that   only  the sets $Z$ of $\binom{t}{2}$ double
points with support  on a star configuration  have $\Delta H_Z = \Delta H_1$.
In the second case, we will show that
the only sets $Z$ of $\binom{t}{2}$ double points and one
simple point that have $\Delta H_Z = \Delta H_2$ are  sets of double points with support
on a star configuration plus one
extra point on one of the lines that define the star configuration.

We start by collecting together some required tools.  The first
result we need is the following theorem (see
\cite{BGM} or \cite{D}).

\begin{theorem}\label{davis}
Let $X\subset \mathbb{P}^2$ be a zero-dimensional subscheme, and
assume that there is a $t$ such that $\Delta H_X(t-1) = \Delta H_X(t) = d$.
Then the degree $t$ components of $I_X$ have a GCD, say $F$, of degree $d$.
Furthermore, the subscheme $W$ of $X$ lying on the curve defined by $F$
(i.e., $I_{W}$ is the saturation of the ideal $(I_X, F)$) has
Hilbert function whose
first difference is given by the truncation $\Delta H_{W}(i) =
\min\{\Delta H_X(i), d\}$.
\end{theorem}

Now we recall some well known results about star configurations.
Given any linear form $L \in R$, we let $\ell$
denote the corresponding line in $\pr^2$. Let
$\ell_1,\ldots,\ell_{t+1}$ be a set of $t+1$ distinct lines in
$\pr^2$ that are three-wise linearly independent (general linear
forms).  In other words, no three lines meet at a point.
 A {\it star configuration} of $\binom{t+1}{2}$ points in
$\pr^2$  is formed from all pairwise intersections of the $t+1$
linear forms.

Geramita, Harbourne, and Migliore have computed the Hilbert function
of double points whose support is a star configuration.
Specifically,

\begin{theorem}[{\cite[Theorem 3.2]{GHM}}]\label{GHMthm}
In $\pr^2$, let $t$  be a positive integer, let $X$ be a star
configuration of $\binom{t+1}{2}$ points, and let $Z=2X$ be a set
of double points whose support is $X$. Then
the first difference of the Hilbert function of $Z$ is
\[\Delta H_{Z}(i) = \left\{
\begin{array}{ll}
i+1 & \mbox{if $0 \leq i \leq t$}\\
t+1 & \mbox{if $t+1 \leq i\leq 2t-1$} \\
0 & \mbox{if $i \geq 2t$.}\\
\end{array}
\right.\]
\end{theorem}

\begin{remark} We point out that the above result can be generalized to $\pr^n$.
\end{remark}


\subsection{Double points on a star configuration}

We consider the valid Hilbert function
\[\Delta H = (\underbrace{1,2,3,\ldots,t}_{t},\underbrace{t+1,\ldots,t+1}_{t}).\]
By Corollary \ref{onlydoubleptsspecial} with $t= \alpha-1$, we already know in this case that
Construction \ref{algorithm}
gives a set of only double points. In this subsection, we will prove that
the configuration produced by our construction is a set of double points
lying on star configuration, and furthermore,
this is the only set of double points whose first difference
is equal to $\Delta H$. Since for $t\leq 2$ all is trivial, we will assume that  $t \geq 3$.

\begin{theorem}\label{doublestar}
Let $t\geq 3$ be an integer, and let $Z \subset \pr^2$ be a set of
$\binom{t+1}{2}$ double points. Then the sequence
\begin{equation}\label{HFstar2}
\Delta H = (\underbrace{1,2,3,\ldots,t}_{t},\underbrace{t+1,\ldots,t+1}_{t})
\end{equation} is the first difference of the Hilbert
function of $Z$ if and only if $Z$ is a set of double points whose
support is a star configuration.
\end{theorem}

\begin{proof}  $(\Leftarrow)$  This follows from Theorem \ref{GHMthm}.

$(\Rightarrow)$
Suppose that $Z$ is a set of
$\binom{t+1}{2}$ double points such that the first difference of
Hilbert function is given by (\ref{HFstar2}), i.e.,
\begin{equation}\label{delta1} \Delta H_{Z}=
\begin{tabular}{ccccccccccccccccccccc}
& & &  &  &  & & $\bullet$&$\bullet$&$\bullet$&\dots&$\bullet$\\
& & &  &  &  &$\bullet$&$\bullet$&$\bullet$&$\bullet$&\dots&$\bullet$&\\
&&&&&&&&\vdots\\
& & &&$\bullet$&\dots&$\bullet$ &$\bullet$ & $\bullet$  &$\bullet$&\dots&$\bullet$\\
& & &$\bullet$&$\bullet$ &\dots  &$\bullet$ & $\bullet$ &$\bullet$&$\bullet$&\dots&$\bullet$\\
&& $\bullet$&$\bullet$ &$\bullet$ &\dots&  $\bullet$ &$\bullet$ &$\bullet$&$\bullet$&\dots&$\bullet$\\
&$\bullet$&$\bullet$ &$\bullet$ &$\bullet$ &\dots  &$\bullet$&$\bullet$ &$\bullet$ &$\bullet$&\dots&$\bullet$ &\\
&$0$&$1$ &$2$ &$3$ & &  &$t$ &$ $ &$ $&&$2t-1$ &\\
\end{tabular}
\\
 t+1  \hbox { \ rows}
 \end{equation}
We want to prove that the support of  $ Z$ must be a star
configuration constructed from  $t+1$ general lines.

Let $I_Z$ be the ideal of $Z$.  We shall sometimes refer to $(I_Z)_t$
as {\it the linear system of all the plane curves of degree $t$ containing $Z$}, since this is
what  the forms in $(I_Z)_t$ correspond to from a geometrical point of view.
From (\ref{HFstar2}) we note that the smallest degree in a minimal
set of generators of $I_Z$ is $t+1$,  the largest degree is $2t$,
and  there is only one curve, say $\mathcal C= \{F=0\}$, in the
linear system defined by $(I_Z)_{t+1}$. Moreover $I_Z$ does not
have new minimal generators until the degree $2t$.

Because $(I_Z)_{2t}$ contains all the forms of type $F\cdot
(x,y,z)^{t-1}$, then the linear system $(I_Z)_{2t}$ cannot be
composed with a pencil, see \cite[pg. 26]{ZariskiBOOK}. We recall that a linear system is composed with a pencil if any of its elements is of the type $\phi_1 \cdot \phi_2 \cdots \phi_n$ where the forms $\phi_i$ are of the type $c_1\psi_1+c_2\psi_2$ for scalars $c_i$ and forms $\psi_i$. Moreover, since $I_Z$ is generated in
degrees $\leq 2t$, the base locus of $(I_Z)_{2t}$ is exactly
$Z$. Hence, $(I_Z)_{2t}$ has no fixed components. By Bertini's
Theorem (for example, see \cite{K}), the general curve of the linear system
$(I_Z)_{2t}$ is integral  (irreducible and reduced). Thus we may
assume that $(I_Z)_{2t}=(G_1,\dots,G_r)$ where $r=\dim_k (I_Z)_{2t}$
and each $\G_i=\{G_i=0\}$ is an integral curve.

For each curve $\G_i,$ consider the intersection $\G_i\cap \C$.
Since each $\G_i$ is integral and $\deg \G_i >
\deg \C$, we  have $\deg (\G_i\cap \C)= \deg \G_i\cdot \deg \C$.
Now each point of $Z$ is a double point of both
$\G_i$ and $\C$.  So, the degree of the scheme  $\G_i\cap \C$ at each point of $Z$ is at least $4$.
Hence
$\deg (\G_i\cap \C) \geq 4{t+1 \choose 2 }=2t(t+1)=\deg \G_i\cdot
\deg \C= \deg (\G_i\cap \C)$.
It follows that $\G_i$ and $\C$ only intersect at the points of $Z$,
and that for each point $P$ in the support of $Z$,
the degree of the scheme  $\G_i\cap \C$ at $P$ is  exactly $4$.

Now observe that the curve $\mathcal C$ is not integral. In fact, $\mathcal C$ has
$\binom{t+1}{2}$ double points, but an integral
curve of degree $t+1$ has at most $\binom{t}{2}$ double points.

We claim  that $\mathcal C$ totally reduces by distinct lines, that is, if
$a\LL$ is an irreducible component of $\C$ (i.e., the polynomial defining
$\LL$  is irreducible) of multiplicity $a$, we
will show that $\LL$ is a line and $a=1$.

Let $P_1$ be a general point on $\LL$. Since $F$ vanishes on
$P_1$, the first difference of the Hilbert function of $Z+P_1$ is the following
$$\Delta H_{Z+P_1}(i) = \left\{
\begin{array}{ll}
i+1 & \mbox{if $0\leq i \leq t$}\\
t +1& \mbox{if $t+1 \leq i\leq 2t-1$} \\
1 & \mbox{if $i = 2t$}\\
0 & \mbox{if $i > 2t$}.\\
\end{array}
\right.$$
To see why this is the case, when we add a point to $Z$,
we have to add a point to the diagram in \eqref{delta1}.
There are  only two places to put a point and maintain
a valid Hilbert function: 1) put a point where $i=t+1$
or 2) put a point where $i= 2t$. In the first case
we get $\Delta H_{Z+P_1}(t+1) = t+2 $, and so we do not have curves
in the linear system defined by $(I_Z)_{t+1}$.  But this
is a contradiction since $P_1$ is a point of $\C$.  So 2) must hold.
We will prove that $\LL$ is a common component for the curves of the linear system
$(I_{Z+P_1})_{2t}$.

Recall that  each $\G_i$ intersects $\LL$  only at the
points of $Z$. If $d$ denotes the degree of $\LL$, then the degree
of $\G_i\cap  a\LL$ is $2t a d$.
Now, consider a curve $\mathcal T=\{T=0\}$ with $T\in (I_{Z+P_1})_{2t}$.
We have that $\deg (\mathcal T \cap a\LL) \geq 2t a d+1$. However $\deg \mathcal T=2t$ and
$\deg a\LL=ad$. So, by
Bezout's Theorem, $\LL$ is a common component for every curve of the linear system
$(I_{Z+P_1})_{2t}$.

Now look at $Z+P_1+P_2$ where $P_2$ is another general point on
$\LL$. Since $\LL$ is a common component for every  curve of
 $ (I_{Z+P_1})_{2t}$, then the first difference of the
Hilbert function of $Z+P_1+P_2$ is given by
$$\Delta H_{Z+P_1+P_2}(i) = \left\{
\begin{array}{ll}
i+1 & \mbox{if $0\leq i \leq t$}\\
t+1 & \mbox{if $t+1 \leq i\leq 2t-1$} \\
1 & \mbox{if $i = 2t$}\\
1 & \mbox{if $i = 2t+1$}\\
0 & \mbox{if $i > 2t+1$}.\\
\end{array}
\right.
$$
So, using Theorem \ref{davis}  with $d=1$, we get that $Z+P_1+P_2$ has a
subscheme of degree $2t+2$ lying on a line $\ell$. But $Z$
imposes independent conditions to the curves of degree $2t-1$, hence
$Z$ has at most $t$ double points with support on a line, and so
$P_1, P_2 \in \ell$ and
  $\deg (\C \cap\ell)=2t+2$. Since  $2t+2>t+1=\deg \C \cdot\ \deg  \ell$, then the
line $\ell$ is a component of $\C$.
Now observe that $P_1$ and $P_2$ are generic points on $\LL$,  so  $\LL$
must be the line $\ell$.

It follows that  every irreducible component of $\C$ is a line, and thus
$$\C=a_1\ell_1+\dots+a_v\ell_v,$$
where the $\ell_i$ are lines and
$a_1+ \cdots+a_v=t+1$.
Now we  will show that $a_i=1$ for all $i$, that is, $\C$ is a union of $t+1$ distinct lines.
First observe that no $a_i$ can be bigger than $2$.  Indeed, if  $a_i > 2$,
then the curve $\C
 \setminus \ell_i$ would be a curve of degree $t$ containing $Z$;
this contradicts the fact that $\C$ is the curve of minimal degree
 containing $Z$. Hence, by this observation, or simply
 by recalling that
$\deg (\G_i\cap \C)$ in every  $P \in Z$ is exactly $4$, the
irreducible components of  $\C$ are simple or double lines.
After relabeling, we can assume
$$\C=2\ell_1+\dots+2\ell_s+\ell_{s+1}+\dots+\ell_{2s+r},$$
where   $2s+r=t+1$.

We observe that the  points of $Z$ lying on the simple lines can lie only
on the intersection with other simple lines.   So there are at most
$\binom{r}{2}$ such points.
Since $\deg (\G_i\cap \C)=4$ for every  $P \in Z$,
the  points of $Z$ on the double lines  cannot lie
on the intersections with other lines.
Moreover, since  $Z$ imposes independent conditions to the curves of $(I_{Z})_{2t}$, on each line $\LL_i$ we have at most $t$ double points,
and so  the number of  points of $Z$ on the double lines is at most
$st$. It follows that   at most $st+\binom{r}{2}$  points of $Z$ lie on $\C$,
that is,
$$|Z|=\binom{t+1}{2}\leq st+\binom{r}{2}.$$

Because $t+1=2s+r$, we have
$$\binom{2s+r}{2}\leq s(2s+r-1)+\binom{r}{2},$$
and from here we get $rs=0$. If $r=0$, then we get
$$\C=2\ell_1+\dots+2\ell_s \ \  \hbox{ where } \ \ 2s=t+1,$$
and $Z$ must have $t$ points on each line $\ell_i$.
By Bezout's Theorem, it follows that the curves of degree $2t-1$ through $Z$
have the lines  $\ell_i$ as fixed components. Removing these $s$
lines from $Z$
we remain with a scheme $Z'$ of $|Z|=\binom{t+1}{2}$ simple points and we get
$$\dim_k (I_Z)_{2t-1}=\dim_k (I_{Z'})_{2t-1-s} \geq
{2t-1-s+2 \choose 2} - {t+1 \choose 2} = \frac{5t^2-4t-1}8.
$$
But $(I_Z)_{2t-1} $ is not defective, hence
$$\dim_k (I_Z)_{2t-1} = {2t-1+2\choose 2} - 3{t+1 \choose 2} =\frac{t^2-t}2.
$$
But $\frac{t^2-t}2 \not\geq \frac{5t^2-4t-1}8$ for any $t$, so we get a contradiction.
Therefore, $s=0$ and
 $\C$ is a union of $t+1$ distinct
lines.
It follows that  the support of $Z$ is a star configuration of $t+1$
lines.
\end{proof}

\begin{remark}
It is easy to see that if $\Delta H_{Z}$ is of type (\ref{HFstar2}),
Construction \ref{algorithm} gives a set of double points on a star configuration. For instance, if
$\Delta H=(1, 2, 3, 4,4,4)$, then $\Delta H^\star = (6,5,4,3)$.  Step 0 and the final step
of Construction \ref{algorithm} are given in Figure \ref{initial}, respectively Figure \ref{final}.
In Figure \ref{final}, the three reduced points near the intersection of $\ell_i$ and $\ell_j$
should be viewed as one double point at $\ell_i \cap \ell_j$.

\begin{figure}[!htbp]
\centering
\mbox{%
\begin{minipage}{0.50\textwidth}
\begin{tikzpicture}[scale=0.60]
\clip(-0.7,-2.5) rectangle (12,4);
\draw [line width=0.8pt] (0.43356525262598167,-0.837313686104764)-- (8.854928961644573,1.7840629353535518);
\draw [line width=0.8pt] (2.435895714524843,-1.8956076501016481)-- (8.113055412043682,3.8363202135355516);
\draw [line width=0.8pt] (6.007747929403006,-2.4789946880528064)-- (6.258448173682883,3.9222182158933747);
\draw [line width=0.8pt] (0.4876183843762104,0.48042166509819506)-- (8.859487147904208,-1.6315379685322804);
\draw (0.5716499374485773,0.4911445340177815) node[anchor=north west] {$\ell_1$};
\draw (0.9054133423577974,-0.5310058935167038) node[anchor=north west] {$\ell_2$};
\draw (2.9914346230404227,-1.365414405789753) node[anchor=north west] {$\ell_3$};
\draw (6.287348246518971,-1.532296108244363) node[anchor=north west] {$\ell_4$};
\begin{scriptsize}
\draw [color=black] (6.421230003179938,-1.0164422459974825) circle (3.5pt);
\draw [color=black] (6.823545190534506,-1.1179337360923736) circle (3.5pt);
\draw [color=black] (7.270032115260305,-1.2305683686457058) circle (3.5pt);
\draw [color=black] (7.726236110334489,-1.3456543129566163) circle (3.5pt);
\draw [color=black] (8.21689477243152,-1.4694320879683898) circle (3.5pt);
\draw [fill=black] (7.66879385079908,1.414846386004017) circle (3.5pt);
\draw [fill=black] (6.766018934556358,1.1338333408907513) circle (3.5pt);
\draw [fill=black] (7.2339071862692155,1.2794761805630945) circle (3.5pt);
\draw [fill=black] (8.109600237525548,1.5520592646636584) circle (3.5pt);
\draw [color=black] (6.658868397026037,2.3681045010267194) ++(-3.5pt,0 pt) -- ++(3.5pt,3.5pt)--++(3.5pt,-3.5pt)--++(-3.5pt,-3.5pt)--++(-3.5pt,3.5pt);
\draw [color=black] (7.054097976498693,2.7671469024766946) ++(-3.5pt,0 pt) -- ++(3.5pt,3.5pt)--++(3.5pt,-3.5pt)--++(-3.5pt,-3.5pt)--++(-3.5pt,3.5pt);
\draw [fill=black] (6.233437899316946,3.2836225437497712) ++(-3.5pt,0 pt) -- ++(3.5pt,3.5pt)--++(3.5pt,-3.5pt)--++(-3.5pt,-3.5pt)--++(-3.5pt,3.5pt);
\draw [color=black] (8.658434177822425,-1.580818616733065) circle (3.5pt);
\draw [fill=black] (8.538567179858813,1.685586796885862) circle (3.5pt);
\draw [color=black] (7.417939024487691,3.1344979639264112) ++(-3.5pt,0 pt) -- ++(3.5pt,3.5pt)--++(3.5pt,-3.5pt)--++(-3.5pt,-3.5pt)--++(-3.5pt,3.5pt);
\draw [color=black] (7.8229006493133815,3.5433662968106594) ++(-3.5pt,0 pt) -- ++(3.5pt,3.5pt)--++(3.5pt,-3.5pt)--++(-3.5pt,-3.5pt)--++(-3.5pt,3.5pt);

\draw [fill=black] (6.214644198463987,2.803756715304235) ++(-3.5pt,0 pt) -- ++(3.5pt,3.5pt)--++(3.5pt,-3.5pt)--++(-3.5pt,-3.5pt)--++(-3.5pt,3.5pt);
\draw [fill=black] (6.250759881302112,3.7259104837710244) ++(-3.5pt,0 pt) -- ++(3.5pt,3.5pt)--++(3.5pt,-3.5pt)--++(-3.5pt,-3.5pt)--++(-3.5pt,3.5pt);
\end{scriptsize}
\end{tikzpicture}
\caption{Initial setup of Construction}\label{initial}
\end{minipage}%
\hskip5pt
\begin{minipage}{0.50\textwidth}
\begin{tikzpicture}[scale=0.60]
\clip(-0.7,-2.5) rectangle (12,4);
\draw [line width=0.8pt] (0.43356525262598167,-0.837313686104764)-- (8.854928961644573,1.7840629353535518);
\draw [line width=0.8pt] (2.435895714524843,-1.8956076501016481)-- (8.113055412043682,3.8363202135355516);
\draw [line width=0.8pt] (6.007747929403006,-2.4789946880528064)-- (6.258448173682883,3.9222182158933747);
\draw [line width=0.8pt] (0.4876183843762104,0.48042166509819506)-- (8.859487147904208,-1.6315379685322804);
\draw (0.5716499374485773,0.4911445340177815) node[anchor=north west] {$\ell_1$};
\draw (0.9054133423577974,-0.5310058935167038) node[anchor=north west] {$\ell_2$};
\draw (2.9914346230404227,-1.365414405789753) node[anchor=north west] {$\ell_3$};
\draw (6.287348246518971,-1.532296108244363) node[anchor=north west] {$\ell_4$};
\begin{scriptsize}
\draw [color=black] (2.6682340415179264,-0.06967919940760992) circle (3.5pt);
\draw [color=black] (2.9668292911611607,-0.14500540576044046) circle (3.5pt);
\draw [color=black] (3.728355738427461,-0.3371146184101969) circle (3.5pt);
\draw [color=black] (4.071834598727522,-0.42376355050041203) circle (3.5pt);
\draw [color=black] (5.876860501684691,-0.8791149125894255) circle (3.5pt);
\draw [fill=black] (4.955063467282757,0.5701246010189326) circle (3.5pt);
\draw [fill=black] (3.191465777581962,0.02115724689686202) circle (3.5pt);
\draw [fill=black] (4.681639817854343,0.4850141145073616) circle (3.5pt);
\draw [fill=black] (6.075467787187878,0.9188806685944131) circle (3.5pt);
\draw [color=black] (4.190170466659754,-0.12440922201586375) ++(-3.5pt,0 pt) -- ++(3.5pt,3.5pt)--++(3.5pt,-3.5pt)--++(-3.5pt,-3.5pt)--++(-3.5pt,3.5pt);
\draw [color=black] (5.073400059331326,0.7673409808955843) ++(-3.5pt,0 pt) -- ++(3.5pt,3.5pt)--++(3.5pt,-3.5pt)--++(-3.5pt,-3.5pt)--++(-3.5pt,3.5pt);
\draw [fill=black] (6.153464568250792,1.2416368238606692) ++(-3.5pt,0 pt) -- ++(3.5pt,3.5pt)--++(3.5pt,-3.5pt)--++(-3.5pt,-3.5pt)--++(-3.5pt,3.5pt);
\draw [color=black] (6.2254648496410425,-0.9670568442880333) circle (3.5pt);
\draw [fill=black] (6.347633063617987,1.0035994525655392) circle (3.5pt);
\draw [color=black] (6.142070865117937,1.8463213681165076) ++(-3.5pt,0 pt) -- ++(3.5pt,3.5pt)--++(3.5pt,-3.5pt)--++(-3.5pt,-3.5pt)--++(-3.5pt,3.5pt);
\draw [color=black] (6.3392372103356065,2.045389798130655) ++(-3.5pt,0 pt) -- ++(3.5pt,3.5pt)--++(3.5pt,-3.5pt)--++(-3.5pt,-3.5pt)--++(-3.5pt,3.5pt);
\draw [fill=black] (6.081056425923274,-0.6071844102353061) ++(-3.5pt,0 pt) -- ++(3.5pt,3.5pt)--++(3.5pt,-3.5pt)--++(-3.5pt,-3.5pt)--++(-3.5pt,3.5pt);
\draw [fill=black] (6.194715293332399,2.294905337611025) ++(-3.5pt,0 pt) -- ++(3.5pt,3.5pt)--++(3.5pt,-3.5pt)--++(-3.5pt,-3.5pt)--++(-3.5pt,3.5pt);
\end{scriptsize}
\end{tikzpicture}
\caption{Output of Construction}\label{final}
\end{minipage}
}
\end{figure}

\end{remark}

\subsection{Double points on a star configuration  plus one point}

In this section  we will investigate when a scheme of one simple
point union $t+1\choose2$ double points has the  same Hilbert
function of one simple point union double points with support on a
star configuration.

\begin{theorem}\label{puntosotto}
Let $t\geq 3$ be an integer,  let $Z \subset \pr^2$ be a set of
${t+1}\choose 2$ double points, and let $P$ be a simple point. Then the sequence
$$
\Delta H = (\underbrace{1,2,3,\ldots,t}_{t},\underbrace{t+1,\ldots,t+1}_{t},1)
$$
is the first difference of the Hilbert function of $Z+P$ if and
only if $Z$ is a set of double points whose support is a star
configuration  of $t+1$ lines and $P$ lies on one of those  lines.
\end{theorem}

\begin{proof} $(\Leftarrow)$  One can compute the Hilbert function
from \cite[Theorem 3.2]{GHM}.

$(\Rightarrow)$
Suppose that $Z+P$ is a set of
${t+1}\choose 2$ double points and a simple point, such that the
first difference of the Hilbert function looks like

\begin{equation}\label{deltaZ+P1} \Delta H_{Z+P}=
\begin{tabular}{ccc ccc ccc ccc ccccccccc}
& & &  &  &  & & $\bullet$&$\bullet$&\dots&$\bullet$&$\bullet$ \\
& & &  &  &  &$\bullet$&$\bullet$&$\bullet$&\dots&$\bullet$&$\bullet$\\
&&&&&&&&\vdots\\
& & &&$\bullet$&\dots&$\bullet$ &$\bullet$ & $\bullet$  &\dots&$\bullet$&$\bullet$\\
& & &$\bullet$&$\bullet$ &\dots  &$\bullet$ & $\bullet$ &$\bullet$&\dots&$\bullet$&$\bullet$\\
&& $\bullet$&$\bullet$ &$\bullet$ &\dots&  $\bullet$ &$\bullet$ &$\bullet$&\dots&$\bullet$&$\bullet$\\
&$\bullet$&$\bullet$ &$\bullet$ &$\bullet$ &\dots  &$\bullet$&$\bullet$ &$\bullet$ &\dots&$\bullet$&$\bullet$ &$\bullet$
\\
&$0$&$1$ &$2$ &$3$ & & &$t$ &$ $ &$ $&& & $2t$ &\\
\end{tabular}
\\ t+1 \hbox{ \ rows}
 \end{equation}
Our goal is to show that the support of  $Z$ must be a star
configuration of $t+1$ lines and a point $P$ that lies on one of those lines.

Consider only the scheme $Z$.
The first difference of the Hilbert
function of  $Z$ can only have one of the following two forms:
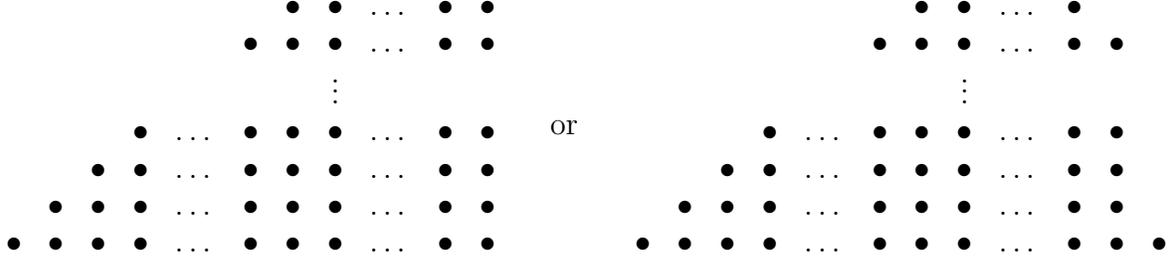
\begin{equation}\label{cases}
\begin{tabular}{ccccccccccccccccccc}
& & &  &  &  & & $\bullet$&$\bullet$&\dots&$\bullet$&$\bullet$ \\
& & &  &  &  &$\bullet$&$\bullet$&$\bullet$&\dots&$\bullet$&$\bullet$\\
&&&&&&&&\vdots\\
& & &&$\bullet$&\dots&$\bullet$ &$\bullet$ & $\bullet$  &\dots&$\bullet$&$\bullet$\\
& & &$\bullet$&$\bullet$ &\dots  &$\bullet$ & $\bullet$ &$\bullet$&\dots&$\bullet$&$\bullet$\\
&& $\bullet$&$\bullet$ &$\bullet$ &\dots&  $\bullet$ &$\bullet$ &$\bullet$&\dots&$\bullet$&$\bullet$\\
&$\bullet$&$\bullet$ &$\bullet$ &$\bullet$ &\dots  &$\bullet$&$\bullet$ &$\bullet$ &\dots&$\bullet$&$\bullet$
\\
\end{tabular}
\text{ \  \ \ or }  \  \begin{tabular}{ccccccccccccccccccc}
& & &  &  &  & & $\bullet$&$\bullet$&\dots&$\bullet$& &\\
& & &  &  &  &$\bullet$&$\bullet$&$\bullet$&\dots&$\bullet$&$\bullet$&\\
&&&&&&&&\vdots\\
& & &&$\bullet$&\dots&$\bullet$ &$\bullet$ & $\bullet$  &\dots&$\bullet$&$\bullet$&\\
& & &$\bullet$&$\bullet$ &\dots  &$\bullet$ & $\bullet$ &$\bullet$&\dots&$\bullet$&$\bullet$&\\
&& $\bullet$&$\bullet$ &$\bullet$ &\dots&  $\bullet$ &$\bullet$ &$\bullet$&\dots&$\bullet$&$\bullet$&\\
&$\bullet$&$\bullet$ &$\bullet$ &$\bullet$ &\dots  &$\bullet$&$\bullet$ &$\bullet$ &\dots&$\bullet$&$\bullet$ & $\bullet$ & \\
\end{tabular}
\end{equation}
To see why, note that $\Delta H_Z$ is constructed from
$\Delta H_{Z+P}$ by removing exactly one point.  The two cases represent the
only two ways to remove a point from \eqref{deltaZ+P1} and still have a valid
Hilbert function.

If $\Delta H_Z$ is of the first type, then by Theorem \ref{HFstar2}, the
support of $Z$ is a star configuration and the curve $\C$
is given by the product of the $t+1$
lines of the star configuration. If
$P$ does not lie on a line of the star configuration, then $P\notin \C$, and
thus $(I_{Z+P})_{t+1}=0$.  Hence the first difference Hilbert
function of  $Z+P$ would have type
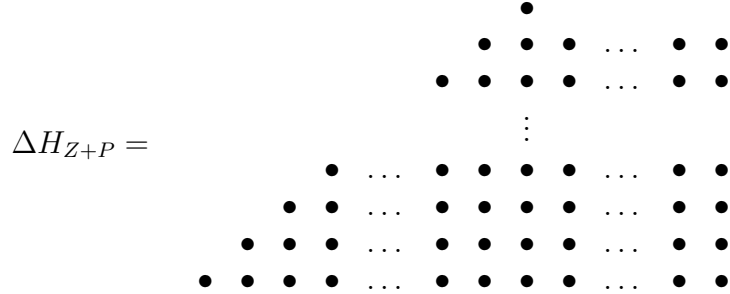
\[  \Delta H_{Z+P}= \begin{tabular}{ccccccccccccccccccccc}
& &  &  &  & &&& $\bullet$&&&& &  &  &  \\
& & &  &  &  & & $\bullet$&$\bullet$&$\bullet$&\dots&$\bullet$&$\bullet$ \\
& & &  &  &  &$\bullet$&$\bullet$&$\bullet$&$\bullet$&\dots&$\bullet$&$\bullet$\\
&&&&&&&&\vdots\\
& & &&$\bullet$&\dots&$\bullet$ &$\bullet$ & $\bullet$  &$\bullet$&\dots&$\bullet$&$\bullet$\\
& & &$\bullet$&$\bullet$ &\dots  &$\bullet$ & $\bullet$ &$\bullet$&$\bullet$&\dots&$\bullet$&$\bullet$\\
&& $\bullet$&$\bullet$ &$\bullet$ &\dots&  $\bullet$ &$\bullet$ &$\bullet$&$\bullet$&\dots&$\bullet$&$\bullet$\\
&$\bullet$&$\bullet$ &$\bullet$ &$\bullet$ &\dots  &$\bullet$&$\bullet$ &$\bullet$ &$\bullet$&\dots&$\bullet$&$\bullet$
\end{tabular}\]
which is different from (\ref{deltaZ+P1}).
So $P$ lies on a line of $\C$, and the conclusion follows.

It suffices to show that the second case cannot occur.
So assume for a contradiction that $\Delta H_Z$ is given by the
second diagram in \eqref{cases}.
Note that the smallest degree in a
minimal set of generators of $I_{Z}$ is
$t+1$ and  there is only one curve, say $\mathcal C= \{F=0\}$,  in the linear system defined by $(I_Z)_{t+1}$.

Now we will show that  the linear system $(I_{Z+P})_{2t}$ has no fixed components.
Suppose for a contradiction that $\mathcal T$ is a fixed irreducible component
of $(I_{Z+P})_{2t}$ and let $Q\in \mathcal T$ be a generic point.
Since
$\dim_k (I_{Z+P+Q})_{2t}=\dim_k (I_{ Z+P})_{2t}$
we have
\[\Delta H_{Z+P+Q}(i) = \left\{
\begin{array}{ll}
i+1 & \mbox{if $i \leq t$}\\
t+1 & \mbox{if $t+1 \leq i\leq 2t-1$} \\
1& \mbox{if $i= 2t$} \\
1& \mbox{if $i= 2t+1$} \\
0 & \mbox{if $i\geq 2t+2$.}\\
\end{array}\\
\right. \]
By Theorem \ref{davis} with $d=1$,  we have  that  $Z+P+Q$ has a subscheme $W$  of degree $2t+2$ lying on a line $\ell$. But $W$  cannot be
contained in $Z$, since $(I_{ Z})_{2t}$ cannot have a scheme of
degree $2t+2$ on a line.  So the scheme $W$ is the intersection
of $\ell$ with $t$  points of $Z $ plus $P$ and $Q$. But $Q$
is generic on $\mathcal T$, thus $\mathcal T$ should be the line
$\ell$. Now consider  the scheme $W \setminus Q \subset
\ell$ which is the union of $t$  points of $Z$ plus $P$.

Since the scheme $W\setminus Q$  has degree $2t+1$, the point $P$
cannot give independent conditions to the curves of the linear
system $(I_{Z+P})_{2t-1}$. If $\Delta H_{Z}$ resembles the second case
of (\ref{cases}), then $\Delta H_{Z+P}$ cannot be of type
(\ref{deltaZ+P1}), since in this case $P$ would impose independent
conditions to the curves of  $(I_{Z+P})_{2t-1}$, a contradiction.
Thus, the linear system $(I_{Z+P})_{2t}$ does not have fixed
components.

Moreover, since $(I_{Z+P})_{2t}$ contains forms of the type
$F\cdot (x,y,z)^{t-1}$, then it cannot be composed
with a pencil, see \cite[pg.26]{ZariskiBOOK}. Using Bertini's Theorem (see \cite{K}), the general
curve of $(I_{Z+P})_{2t}$ is reduced and irreducible. Let
$\mathcal G$ be such a general integral
 curve. We have that $\deg \mathcal G \cdot \deg
\C=2t(t+1)$. But since $\deg (\mathcal G \cap \C)\geq
4\frac{t(t+1)}{2}+1$, the curve $\mathcal G$ would contain a component of
$\C$, thus giving a contradiction.
\end{proof}

Finally we show that if we add a multiplicity \lq\lq on the
top\rq\rq  of (\ref{delta1}),  a statement similar  to that of
Theorem \ref{puntosotto} does not hold.  More precisely
let $t\geq 3$ be an integer,  and  let $Z \subset \pr^2$ be a set of
${t+1}\choose 2$ double points.  Suppose that $P$ is a simple point
such that the first difference Hilbert function of $Z+P$ has the
form
\begin{equation} \label{deltaZ+P2}
  \Delta H_{Z+P}= \begin{tabular}{ccccccccccccccccccccc}
  & &  &  &  & &&& $\bullet$&&&& \\
& & &  &  &  & & $\bullet$&$\bullet$&$\bullet$&\dots&$\bullet$&$\bullet$ \\
& & &  &  &  &$\bullet$&$\bullet$&$\bullet$&$\bullet$&\dots&$\bullet$&$\bullet$\\
&&&&&&&&\vdots\\
& & &&$\bullet$&\dots&$\bullet$ &$\bullet$ & $\bullet$  &$\bullet$&\dots&$\bullet$&$\bullet$\\
& & &$\bullet$&$\bullet$ &\dots  &$\bullet$ & $\bullet$ &$\bullet$&$\bullet$&\dots&$\bullet$&$\bullet$\\
&& $\bullet$&$\bullet$ &$\bullet$ &\dots&  $\bullet$ &$\bullet$ &$\bullet$&$\bullet$&\dots&$\bullet$&$\bullet$\\
&$\bullet$&$\bullet$ &$\bullet$ &$\bullet$ &\dots  &$\bullet$&$\bullet$ &$\bullet$ &$\bullet$&\dots&$\bullet$&$\bullet$
\\
&$0$&$1$ &$2$ &$3$ & & &$t$ &$ $ &$ $&&&$2t-1$&  &\\
\end{tabular}
\\ t+1 \hbox{ \ rows}
 \end{equation}
Obviously, if the support of $Z$ is a star configuration  of $t+1$
lines and $P$ is a generic simple point, then the  first
difference of the Hilbert function of $Z+P$ is  of type
(\ref{deltaZ+P2}), but the converse
is not true, as we show in the following example.

\begin{example}
Consider a scheme $Z$ of ${t+1}\choose 2$ double points,
whose support, except one double point, say $2Q$,
are the points of a star configuration of the $t+1$ general lines
$\ell_1,\dots$,  $\ell_{t+1}$.
More precisely, the points of $Z$ lie  on the intersections  $\ell_i \cap \ell_j$, for all $i \neq j$, except for $(i,j) = (1,2)$.
Let the simple point $P$ be a general point on  $\ell_{1}$ and suppose
 $Q$ is a general point on  $\ell_2$ (see Figure \ref{fig11}).  Note
that in Figure \ref{fig11}, all the points of intersection are double points.

In order to prove that the  first difference of the Hilbert
function of $Z+P$ is of type (\ref{deltaZ+P2}),  it is enough to
prove that $\dim_k (I_{Z+P})_{t+2}=2$ and  $(I_{Z+P})_{2t-1}$ is not defective.
By Bezout's Theorem, the $t$ lines $\ell_{2}, \dots, \ell_{t+1}$
are fixed components for the curves of the two linear systems
$(I_{Z+P})_{t+2}$ and  $(I_{Z+P})_{2t-1}$.  Hence we have
$$\dim_k (I_{Z+P})_{t+2}= \dim_k (I_{X})_{2} ~~~
\mbox{and}~~~ \dim_k (I_{Z+P})_{2t-1}=\dim_k(I_{X})_{t-1},$$
where $X$ is a union of  $t$  simple points on the line $\ell_{1}$
and the point $Q \in \ell_2$.
Since $2 <t$, in order to compute $\dim (I_{X})_{2}$,
we may remove the line $\ell_{1}$ and we get
$\dim_k (I_{Z+P})_{t+2}= \dim_k (I_{Q})_{1}=2.$
Since $X$ imposes independent conditions to the curve of degree $t-1$ we have
$$\dim_k (I_{X})_{t-1}= {t-1+2\choose 2} - t-1={t\choose 2} -1,$$
which is the expected dimension of $ (I_{Z+P})_{2t-1}$, and we are
done.
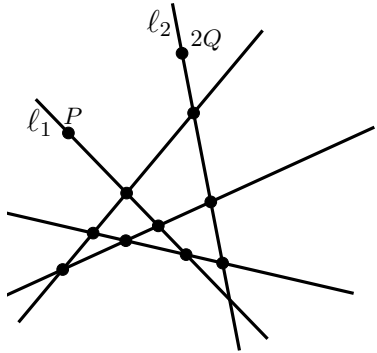
\begin{figure}[!h]
\centering
\mbox{%
\begin{minipage}{0.40\textwidth}
\begin{tikzpicture}[scale=0.55]
\clip(-1.7,-3.5) rectangle (12,6);
\draw [line width=1.2pt] (-2.16,-2.36)-- (7.16,1.9);
\draw [line width=1.2pt] (-1.44,-2.82)-- (4.46,4.24);
\draw [line width=1.2pt] (-2.5,0.)-- (6.66,-2.12);
\draw [line width=1.2pt] (-1.0238993710691824,2.5763052208835338)-- (4.58,-3.22);
\draw [line width=1.2pt] (2.279245283018868,4.895582329317269)-- (3.9,-3.5);
\draw (-1.5,2.6) node[anchor=north west] {$\ell_1$};
\draw (1.4,5.006024096385542) node[anchor=north west] {$\ell_2$};
\draw (1.3,-3.7) node[anchor=north west] {Figure 1};
\begin{scriptsize}
\draw [fill=black] (1.1717603989446856,0.30525905365245454) circle (4.0pt);
\draw [fill=black] (2.7912842009522194,2.243197704868249) circle (4.0pt);
\draw [fill=black] (3.2063994309236463,0.09288214331917732) circle (4.0pt);
\draw [fill=black] (1.153340864239537,-0.845533038448452) circle (4.0pt);
\draw [fill=black] (0.36292727423501103,-0.6625988887967494) circle (4.0pt);
\draw [fill=black] (-0.3729706973038376,-1.5431818852483203) circle (4.0pt);
\draw [fill=black] (2.6103764970817362,-1.1827508923376944) circle (4.0pt);
\draw [fill=black] (3.492051147780308,-1.386806597521207) circle (4.0pt);
\draw [fill=black] (-0.23466331368864224,1.7599713117627585) circle (4.0pt);
\draw[color=black] (-0.12452830188679241,2.1676706827309236) node {$P$};
\draw [fill=black] (2.512396010393332,3.687851075288592) circle (4.0pt);
\draw[color=black] (3.1,4) node {$2Q$};
\draw [fill=black] (1.937725761225849,-0.48700517780878594) circle (4.0pt);
\end{scriptsize}
\end{tikzpicture}
\end{minipage}
}
\caption{All but one double point on a star configuration plus one double
point and one simple point}\label{fig11}
\end{figure}

\end{example}

\section{Final remarks}

In this paper we presented an algorithm that, given a valid Hilbert function $H$ for a
zero-dimensional scheme, will produce a scheme consisting of double and simple points having Hilbert function $H$.

We know that for some special $H$ (for example, see Theorem \ref{HFofdoublepts}) our algorithm will produce a set consisting of {\em only} double points.
Furthermore, in  Section \ref{specialconfigsection}
we showed that for one family of valid Hilbert functions
$H$, not only does our algorithm produce a scheme with the
maximal possible number of double points, our algorithm produces
the only possible configuration of double points with this
Hilbert function.

However, there are $H$ for which our algorithm does not perform well. Consider, for example, when $H$ is the Hilbert function of double points with collinear support; our algorithm will produce a set with just {\em one} double point! Thus it is natural to ask how well our algorithm performs.

The major obstacle in answering this question is the following: given a Hilbert function $H$ of a degree $3t$ zero-dimensional scheme, we do not know the maximal number of double points that the scheme can possess. Of course $t$ gives an upper bound, but this bound might not be sharp. Ideally we would like to compare this unknown number with the number of double points that our algorithm produces for $H$ and possibly make some asymptotic estimate.

This problem will be the object of further investigations, but we can already present an interesting result. Consider the scheme consisting of $t$, generic double points and let $H$ be its Hilbert function. It is well known (e.g., see \cite{alexanderhirschowitz}) that, with the exceptions $t=2$ and $t=5$, $H(i)=\min\left\{{i+2\choose 2}, 3t\right\}$. The following result describes the asymptotic behavior of our algorithm for this $H$.

\begin{proposition} Let $H$ be the Hilbert function of $t$ generic double points, and let $s(t)$ be the number of double points produced by our algorithm with input $H$.  Then
\[
\lim_{t\rightarrow+\infty} \frac{s(t)}{t}=\frac{3}{4}.
\]
\end{proposition}

\begin{proof}
For each positive integer $t$, we choose $b$ and $0\leq\epsilon\leq b+1$ such that $3t={b+2 \choose 2}+\epsilon$; note that $b$ is uniquely determined by $t$ and viceversa.
We consider the case $b$ odd, and a similar argument applies in the case $b$ even.
With this notation we have that, for $t\geq 6$,
\[
\Delta H=(1,2,\ldots,b+1,\epsilon).
\]
Moreover, it is easy to see that, applying our algorithm to $\Delta H$ produces the same result when applying our algorithm to the length $b+1$ sequence
\[
\Delta H_1=(1,2,\ldots,\frac{b+1}{2},\underbrace{\frac{b+3}{2},\ldots,\frac{b+3}{2}}_{\frac{b+1}{2}}).
\]
As shown is Section \ref{specialconfigsection}, we obtain a set of
\[
s(b)=\frac{1}{8}(b+1)(b+3)
\]
double points.
By a change of variables and using the bound $\epsilon\leq b+1$ we get
\[
\lim_{t\rightarrow+\infty}\frac{s(t)}{t}=
\lim_{d\rightarrow+\infty} \frac{s(b)}{\frac{1}{3}{b+2\choose 2}+\frac{\epsilon}{3}}=\frac{3}{4},
\]
and this complete the proof.
\end{proof}



\begin{thebibliography}{99}
\bibitem{C}
  J. Abbott, A. Bigatti, L.
  Robbiano, CoCoA: a system for doing Computations in Commutative
  Algebra. Available at {\tt http://cocoa.dima.unige.it}

\bibitem{alexanderhirschowitz}
J. Alexander, A. Hirschowitz,
Polynomial interpolation in several variables.
J. Alg. Geom. {\bf 4} (2) (1995), 201--222.

\bibitem{BGM}
A. Bigatti, A. V. Geramita, J. Migliore,
Geometric consequences of extremal behavior in a theorem of
Macaulay.
Trans. Amer. Math. Soc. {\bf 346} (1994), 203--235.


\bibitem{CGvT2014}
E. Carlini, E. Guardo, A. Van Tuyl,
Star configurations on generic hypersurfaces.
J. Algebra {\bf 407} (2014), 1--20.

\bibitem{CvT2011}
E. Carlini, A. Van Tuyl,
Star configuration points and generic plane curves.
Proc. Amer. Math. Soc. {\bf 139} (2011), 4181--4192.

\bibitem{sundials}
E. Carlini, M.V. Catalisano, A.V. Geramita,
3-dimensional sundials.
Cent. Eur. J. Math. {\bf 9} 5 (2011) 949--971.

\bibitem{marvialessandro}
M.V. Catalisano, A. Oneto,
Tangential varieties of Segre-Veronese surfaces are never defective.
To appear in: Rev. Mat. Complut.

\bibitem{CHT}
S. Cooper, B. Harbourne, Z. Teitler,
Combinatorial bounds on Hilbert functions of fat points in projective space.
J. Pure Appl. Algebra {\bf 215} (2011), 2165--2179.

\bibitem{D} E. Davis,
Complete intersections of codimension 2 in $\mathbb P^r$:
The Bezout-Jacobi-Segre theorem revisited.
Rend. Sem. Mat. Univ. Politec. Torino {\bf 43} (1985), 333--353.

\bibitem{GHM} A.V. Geramita, B. Harbourne, J. Migliore,
Star configurations in $\mathbb{P}^n$.
J. Algebra {\bf 376} (2013), 279--299.

\bibitem{GHS:1}
A.V. Geramita, T. Harima, Y.S. Shin,
An alternative to the Hilbert function for the ideal of a
finite set of points in $\mathbb{P}^n$.
Illinois J. Math. {\bf 45} (2001), 1--23.


\bibitem{GMR} A.V. Geramita, P. Maroscia and L. Roberts,
The Hilbert function of a reduced $k$-algebra.
J. London Math. Soc. {\bf 28} (1983), 443--452.


\bibitem{Gi} A. Gimigliano,
Our thin knowledge about fat points.
{\it The Curves Seminar at Queen's, Vol. VI (Kingston, ON, 1989)},
Exp. No. B, 50 pp.,
Queen's Papers in Pure and Appl. Math., {\bf 83},
Queen's Univ., Kingston, ON, 1989.

\bibitem{H} B. Harbourne, Problems and Progress: A survey on fat points in
$\mathbb{P}^2$.
{\it Zero-dimensional schemes and applications (Naples, 2000)}, 85–132,
Queen's Papers in Pure and Appl. Math., {\bf 123},
Queen's Univ., Kingston, ON, 2002.
{\tt http://www.math.unl.edu/$\sim$bharbourne1/srvy9-12-01.pdf}

\bibitem{K} S.L. Kleiman, Bertini and his two fundamental
theorems. Rendiconti del circolo Matematico di Palermo, Serie II
-- Supplemento (1997). {\tt alg-geom/9704018}

\bibitem{KR}
  M. Kreuzer, L. Robbiano,
  {\it Computational commutative algebra 2}.
  Springer-Verlag, Berlin, 2005.

\bibitem{RR}
L.G. Roberts, M. Roitman,
On Hilbert functions of reduced and of integral algebras.
J. Pure Appl. Algebra {\bf 56} (1989), 85--104.

\bibitem{ZariskiBOOK}
O. Zariski,
Algebraic surfaces. With appendices by S. S. Abhyankar, J. Lipman and D. Mumford. Preface to the appendices by Mumford.
Reprint of the second (1971) edition. Classics in Mathematics. Springer-Verlag, Berlin, 1995.


\end{thebibliography}
\end{document}